\documentclass[12pt,a4paper]{article}

\usepackage[a4paper,bindingoffset=0.2in,%
            left=1in,right=1in,top=1in,bottom=1in,%
            footskip=.25in]{geometry}\usepackage[utf8]{inputenc}
\usepackage[authoryear]{natbib}

\usepackage[english]{babel}
\usepackage{amsthm,amssymb,amsfonts,amsmath}
\usepackage[authoryear]{natbib}
\usepackage{dsfont}
\usepackage{color,fancybox,graphicx}
\usepackage{url}
\usepackage{psfrag}
\usepackage{color}
\usepackage{pifont}
 \usepackage{lmodern}
\usepackage{latexsym,pstricks} 
\usepackage{subfig}

\usepackage{algorithmic}
\usepackage{algorithm}

\DeclareMathOperator*{\argmax}{arg\,max}
\DeclareMathOperator*{\corr}{Corr}

\usepackage[colorlinks=true, linkcolor=red, citecolor = blue]{hyperref}
\usepackage{mathrsfs}

\usepackage{pst-tree}
\usepackage{fancyhdr}
\usepackage{etoolbox}
\usepackage{setspace}
\usepackage{enumitem}

\newcommand{\E}{\mathds{E}}
\newcommand{\V}{\mathds{V}}
\renewcommand{\P}{\mathds{P}}

\newcommand{\bX}{{\bf X}}
\newcommand{\bx}{{\bf x}}

\newcommand{\empmdi}{\widehat{\textrm{MDI}}}
\newcommand{\mdi}{\textrm{MDI}^{ \star}}
\newcommand{\tree}{\mathcal{T}}
\newcommand{\unifbeta}{\mathcal{U}^{\otimes 2^{\beta}}}

\newtheorem{remark}{Remark}
\newtheorem{theorem}{Theorem}

\newtheorem{lemme}{Lemma}
\newtheorem{techlemme}{Technical Lemma}
\newtheorem{model}{Model}
\newtheorem{proposition}{Proposition}

\newtheorem{corollary}{Corollary}

\newtheoremstyle{break}  % follow `plain` defaults but change HEADSPACE.
  {\topsep}   % ABOVESPACE
  {\topsep}   % BELOWSPACE
  {\itshape}  % BODYFONT
  {0pt}       % INDENT (empty value is the same as 0pt)
  {\bfseries} % HEADFONT
  {}         % HEADPUNCT
  {5pt plus 1pt minus 1pt}  % HEADSPACE. `plain` default: {5pt plus 1pt minus 1pt}
  {}          % CUSTOM-HEAD-SPEC
\theoremstyle{break}

\author{}
\title{Trees, forests, and impurity-based variable importance in regression}
\date{}
\begin{document}

\maketitle

\noindent {\bf Erwan Scornet}\\
{\it Centre de Math\'ematiques Appliqu\'ees, Ecole Polytechnique, CNRS,\\ 
	Institut Polytechnique de Paris, Palaiseau, France}\\
\href{mailto:erwan.scornet@polytechnique.edu}{erwan.scornet@polytechnique.edu}

\begin{abstract}
Tree ensemble methods such as random forests \citep{Br01} 
%and Boosting \cite{chen2016xgboost,ke2017lightgbm} 
are very popular to handle high-dimensional tabular data sets, notably because of their ability to detect sparse signals and their resulting good predictive accuracy. However, when machine learning is used for decision-making problems, settling for the best predictive procedures may not be reasonable since enlightened decisions require to understand the phenomena underlying the data, which is accessible only with an in-depth comprehension of the algorithm prediction process. Unfortunately, random forests are not intrinsically interpretable since their prediction results from averaging several hundreds of decision trees.  
A classic approach to gain knowledge on this so-called black-box algorithm is to compute variable importances, that are employed to assess the predictive impact of each input variable. Variable importances are then used to rank or select variables and thus play a great role in data analysis. Mean Decrease Impurity (MDI) is one of the two variable importance measures in random forests. However, there is no theoretical justification to use MDI: we do not even know what this indicator estimates. 
%MDI iFew works focus on \es{Until recently, there was no justification to use variable importances in such way. In particular, we do not even know what the Mean Decrease Impurity (MDI, one of the two importance measures in random forests) estimates!}
%
In this paper, we analyze MDI and prove that if input variables are independent and in absence of interactions, MDI provides a variance decomposition of the output, where the contribution of each variable is clearly identified. We also study models exhibiting dependence between input variables or interaction, for which the variable importance is intrinsically ill-defined. 
%Our analysis shows that there may exist some benefits to use a forest compared to a single tree.           
\medskip 

\noindent \emph{Index Terms} --- Random forests, variable importance, Mean Decrease Impurity.
\medskip

\noindent {2010 Mathematics Subject Classification}: 62G08, 62G20.

\end{abstract}

\section{Introduction}

Decision trees \citep[see][for an overview]{loh2011classification} can be used to solve pattern recognition tasks in an interpretable way: in order to build a prediction, trees ask to each observation a series of questions, each one being of the form ``Is variable $X^{(j)}$ larger than a threshold $s$?'' where $j,s$ are to be determined by the algorithm. Thus the prediction for a new observation only  depends on the sequence of questions/answers for this observation. One of the most popular decision trees is the CART procedure \citep[Classification And Regression Trees,][]{BrFrOlSt84}, which unfortunately suffers from low accuracy and intrinsic instability: modifying slightly one observation in the training set can change the entire tree structure together with the predicted values in some areas of the input space. 

To overcome this last issue, and also improve their accuracy, ensemble methods, which grow many base learners and aggregate them to predict, have been designed. Random forests, introduced by \citet{Br01}, are among the most famous tree ensemble methods. They proceed by growing trees based on CART procedure \citep[][]{BrFrOlSt84} and randomizing both the training set and the splitting variables in each node of the tree. Breiman's \citeyearpar{Br01} random forests have been under active investigation during the last two decades mainly because of their good practical performance and their ability to handle high-dimensional data sets. They are acknowledged to be state-of-the-art methods in fields such as genomics \citep[][]{Qi12} and pattern recognition \citep[][]{RoRiRaOrTo08}, just to name a few.
%Moreover, they are easy to run since they only depend on few parameters which are easily tunable \citep[][]{LiWi02, GePoTu08}. 
%The ease of the implementation of random forests algorithms is one of their key strengths and has greatly contributed to their widespread use.
%Indeed, a proper tuning of the different parameters of the algorithm is not mandatory to obtain a plausible prediction, making random forests a turn-key solution to deal with large, heterogeneous data sets in many fields \citep[see, e.g.,][]{FCeBaAm14}.
%
If empirical performances of random forests are not to be demonstrated anymore \citep{fernandez2014we}, their main flaw relies on their lack of interpretability because their predictions result from averaging over a large number of fully-grown trees (typically several hundreds):  each tree may be  interpreted, due to its recursive structure, but random forests remain an obscure black-box procedure.

Since interpretability is a concept difficult to define precisely, people eager to gain insights about the driving forces at work behind random forests predictions often focus on variable importance, a measure of the influence of each input variable to predict the output. 
%Unfortunately, there does not exist a unique measure of variable importance since it strongly depends on the task for which it is used afterwards. For example, $(i)$ selecting the smallest model that has at least a prespecified accuracy, and $(ii)$ selecting all variables that are marginally related to the output are two different questions which necessitate in general to calculate two different variable importances. Following this paradigm, \cite{genuer2015vsurf} algorithm outputs not a single set of relevant variables but two sets: a large one supposed to contain all relevant variable and a smaller one that tries to build an accurate model. 
%From a theoretical perspective, the story is less conclusive due to the intrinsic complexity of the original algorithm, which mixes subsampling steps and variable selection procedures. The recent attempts to break the black-box nature of random forests focused on their estimation properties. \es{For example, Scornet, Wager Hooker...} \citep[see][for a theoretical review on random forests]{BiSc16} .  Trying to uncover the complex nature of data requires to dig into the estimator building process in order to explain how the algorithm makes its prediction. In particular, the variable importances produced by random forests are of particular interest since it sheds lights on the set of variables that are likely to be related to the outcome. 
%
In Breiman's \citeyearpar{Br01} original random forests, there exist two importance measures: the Mean Decrease Impurity \citep[MDI, or Gini importance, see][]{breiman2002manual}, which sums up the gain associated to all splits performed along a given variable; and the Mean Decrease Accuracy \citep[MDA, or permutation importance, see][]{Br01} which shuffles entries of a specific variable in the test data set and computes the difference between the error on the permuted test set and the original test set. Because of its very definition, MDI is an importance measure that can be computed for trees only, since it strongly relies on the tree structure, whereas MDA is an instantiation of the permutation importance that can be used for any predictive model. Both measures are used in practice even if they possess several major drawbacks.

MDI is known to favor variables with many categories \citep[see, e.g., ][]{strobl2007bias, nicodemus2011letter}. Even when variables have the same number of categories, MDI exhibits empirical bias towards variables that possess a category having a high frequency \citep{nicodemus2011letter, boulesteix2011random}. MDI is also biased in presence of correlated features \citep[][]{nicodemus2009predictor}.
A promising way to assess variable importance in random forest consists in designing new tree building process or new feature importances as in the \texttt{R} package \texttt{party} \citep{StBoKnAuZe08,strobl2009party, RPackage} or in \cite{li2019debiased, zhou2019unbiased}. 
On the other hand, MDA seems to exhibit less bias than MDI but tends to overestimate correlated features \cite{StBoKnAuZe08}. Besides, its scale version \citep[the default version in the \texttt{R} package \texttt{randomForest},][]{randomForestPackage} depends on the sample size and on the number of trees, the last being an arbitrary chosen parameter \citep{strobl2008danger}. The interested reader may refer to \cite{genuer2008random} and \cite{genuer2010variable} for an extensive simulation study about the influence of the number of observations, variables, and trees on MDA together with the impact of correlation on this importance measure.  
%This unfairness depends on the subsampling procedure: \citet{strobl2007bias} show numerically that using bootstrap subsamples tends to increase the inclination of the forest to choose variables with many categories, compared to subsampling without replacement. 
%Gini importance is also biased in presence of correlated features while permutation variable importance seems to be more insensitive to correlations \citep[][]{StBoKnAuZe08, nicodemus2009predictor}. 
%
%Even when the number of categories across variables is equal, Gini importance exhibit empirical bias towards variables that possess a category having a high frequency \cite{nicodemus2011letter, boulesteix2011random}. Gini importance is also biased in presence of correlated features while permutation variable importance seems to be more insensitive to correlations \citep[][]{StBoKnAuZe08, nicodemus2009predictor}. However, the scale version of MDA depends on the sample size and on the number of trees, the last being an arbitrary chosen parameter \citep[see][]{strobl2008danger}
Despite all their shortcomings, one great advantage of MDI and MDA is their ability to take into account interaction between features even if unable to determine the part of marginal/joint effect \citep{wright2016little}. 

%Similar extensive study about MDI are to be found in and in \citep[see also][for a genetic application]{archer2008empirical}.
%For a practical and methodological overview of trees and forests, the interested reader can refer to \cite{strobl2009introduction}.

From a theoretical perspective, there are only but a few results on variable importance computed via tree procedures. The starting point for theory is by \cite{Is07} who studies a modified version of permutation importance and derives its asymptotic positive expression, when the regression function is assumed to be piecewise constant. A scaled population version of the permutation importance is studied by \cite{ZhZeKo12} who establish rate of convergence for the probability that a relevant feature is chosen in a modified random forest algorithm.
The exact population expression of permutation importance is computed by \cite{gregorutti2017correlation} for a Gaussian linear model.

Regarding the MDI, there are even fewer results. \cite{LoWeSuGe13} study the population MDI when all features are categorical. In this framework, they propose a decomposition of the total information, depending on the MDI of each variable and of interaction terms. They also prove that the MDI of an irrelevant variable is zero and that adding irrelevant variables does not impact the MDI of relevant ones. Even if these results may appear to be very simple, the proofs are not straightforward at all, which explains the few results on this topic. This work was later extended by \cite{sutera2016context} to the case of context-dependent features. 
%The most recent and promising result is by \cite{kazemitabar2017variable} who exhibit finite sample concentration under sparse setting (correlated and uncorrelated design) but they consider only the first split of the tree to compute the MDI, henceforth not relying on the intrinsic recursive structure of trees. 

Whereas recent works tackled the question of (in)consistency of the MDA \Citep{ramosaj2019asymptotic, benard2021mda}, there exists no general result establishing the consistency of MDI toward population quantities when computed with the original random forest algorithms: all existing results about MDI focus on modified random forests version with, in some cases, strong assumptions on the regression model. Therefore, there are no guarantees that using impurity-based variable importance computed via random forests is suitable to select variables, which is nevertheless often done in practice.

%Application : \cite{baca2007variables}: application of biased variable importance criteria and test 

\paragraph{Agenda.} In this paper, we focus on the regression framework and derive theoretical properties about the MDI used in Breiman's \citeyearpar{Br01} random forests. Section~\ref{sec:notations} is devoted to notations and  describes tree and forest construction, together with MDI calculation. 
The main results are gathered in Section~\ref{sec:theoretical_results}. We prove that both population MDI and empirical MDI, computed with Breiman's random forests, can be seen as a percentage of explained variance and are directly related to the quadratic risk of the forest, a result already proved by \cite{LoWeSuGe13} and \citet{sutera2016context} for the population MDI in presence of categorical features. We also derive the expression of the sum of MDIs for two groups of variables that are independent and do not interact, which is at the core of our analysis. The analysis highlights that fully-developed random forests (whose tree leaves contain exactly one observation) produce inaccurate variable importances and should not be used for that purpose.  
In Section~\ref{sec:additive_models}, we establish that for additive regression models with independent inputs, the MDI of each variable takes a very simple, interpretable form, which is the variance of each univariate component $m_j(X^{(j)})$. We prove that empirical MDI computed with Breiman's forest targets this population quantity, which is nothing but a standard way to decompose the variance of the output. Hence, in the absence of input dependence and correlation, MDI computed with random forests provides a good assessment of the variable importance. To the best of our knowledge, it is the first result proving the good properties of the empirical MDI computed with Breiman's forests. 
To move beyond the additive models and understand the impact of interactions, we study a multiplicative model in Section~\ref{sec:interactions}. In this model, MDI is ill-defined and one needs to aggregate many trees to obtain a consistent MDI value. Averaging trees is known to be of great importance to improve accuracy but it turns out to be mandatory to obtain consistent variable importance measure in the presence of interactions. 
The impact of dependence between input variables is studied in Section~\ref{sec:corr}. We stress out that correlated variables are more likely to be chosen in the splitting procedure and, as a consequence, exhibit a larger MDI than independent ones. As for the presence of interactions, this highlights the importance of computing MDI by averaging many trees. 
%Experiments presented in Section~\ref{sec:experiments} illustrate the properties of MDI computed via trees and forests in a finite-sample regime. Additional experiments are gathered in Section~\ref{sec:additional_figures} while 
Proofs are postponed to Section~\ref{sec:proofs}. Experiments on MDI can be found on the ArXiv version of this paper.

%New variable importance measure:
%\begin{itemize}
%	\item \cite{strobl2007unbiased} propose a threshold to discriminate between relevant and irrelevant variables in random forests
%	
%	\item \cite{hapfelmeier2013new}
%	
%	\item \citet{janitza2013auc} AUC based importance measure developed to deal with imbalanced data sets. 
%	
%\end{itemize}

%\paragraph{Remarks - to be incorporated or not}
%\begin{itemize}
%	\item  Le MDI théorique ne dépend pas du niveau du bruit. Ce n'est pas le cas du MDA (preuve empirique?)
%	\item  Many intuitions are true in additive models with independent input but do not hold anymore in models with dependence or interaction
%\end{itemize}

\section{Notation and general definition}
\label{sec:notations}

We assume to be given a data set $\mathcal{D}_n = \{(\bX_i, Y_i), i = 1, \hdots , n \}$ of \textit{i.i.d.} observations $(\bX_i, Y_i)$ distributed as the generic pair $(\bX, Y)$ where $ \bX \in [0,1]^d$ and $Y \in \mathds{R}$ with $\mathds{E}[Y^2] < \infty$. Our aim is to estimate the regression function $m : [0,1]^d \to \mathds{R}$, defined as $m(\bX) = \E [Y|\bX]$, using a random forest procedure and to identify relevant variables, i.e. variables on which $m$ depends.

\begin{remark}
Note that we confine the analysis to the regression setting with (bounded) continuous input variables. Extending our work to the classification setting is not straightforward, since our main result relies on the additive regression framework which does not have any equivalent in classification. Assuming that all input variables are continuous is a mild and common hypothesis in the random forest literature \citep[see, e.g., ][just to name a few]{ArGe14, wager2015adaptive, zhou2019unbiased} which drastically eases the  mathematical analysis.
\end{remark}

\subsection{CART construction and empirical MDI}

\label{subsec:emp_cart}

CART \citep{BrFrOlSt84} is the elementary component of random forests \citep{Br01}. CART construction works as follows. Each node of a single tree $\tree$ is associated with a hyper-rectangular cell included in $[0,1]^d$. The root of the tree is $[0,1]^d$ itself and, at each step of the  tree construction, a node (or equivalently its corresponding cell) is split in two parts. The terminal nodes (or leaves), taken together, form a partition of $[0,1]^d$. An example of such tree and partition is depicted in Figure~\ref{figure1}, whereas the full procedure is described in Algorithm~\ref{alg:split-cell}.

\begin{figure}[h!!]
	\begin{minipage}[c]{.46\linewidth}
		\includegraphics[scale=0.5]{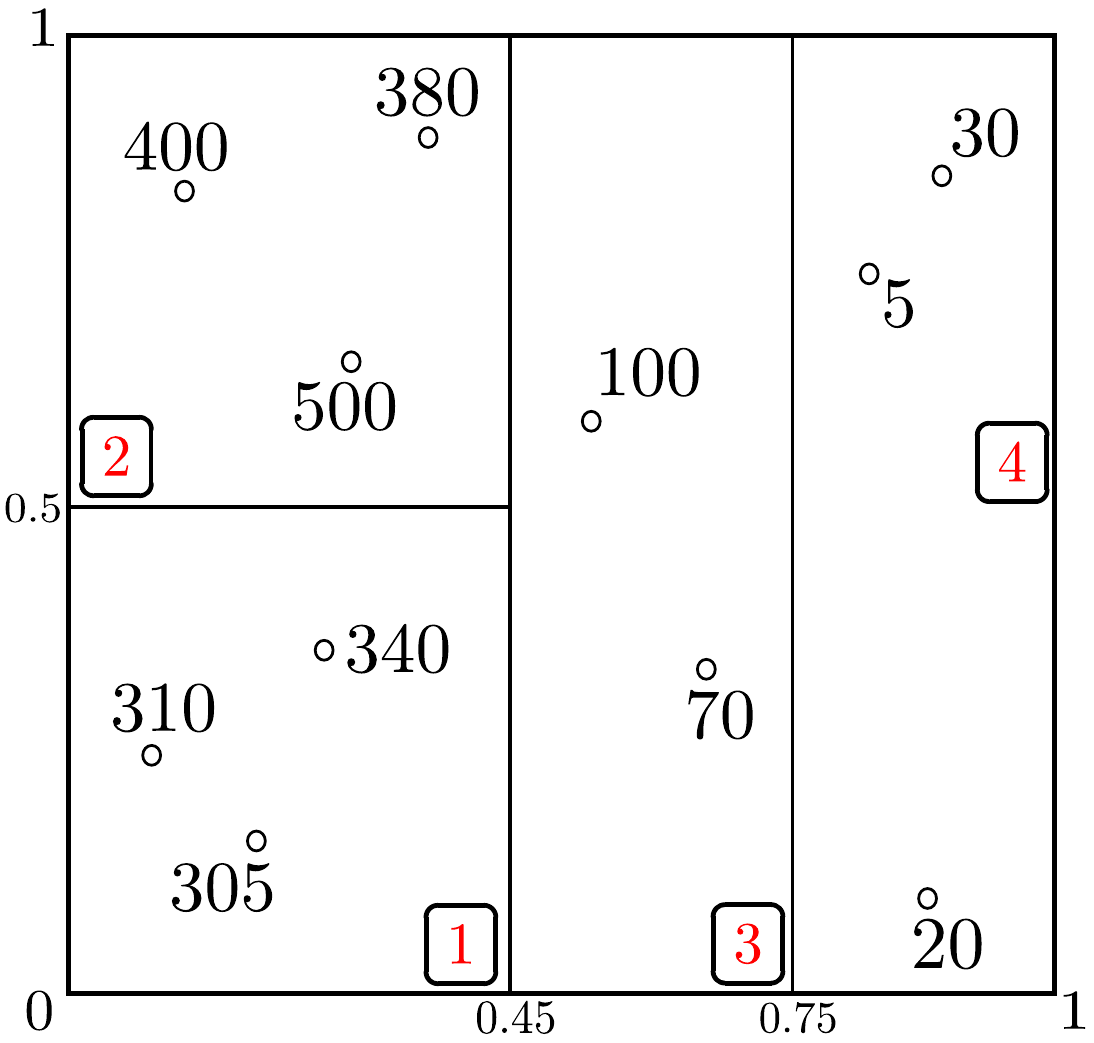}
	\end{minipage} 
	\hspace{-0.5cm}
	\begin{minipage}[l]{.45\linewidth}
		\includegraphics[scale=0.5]{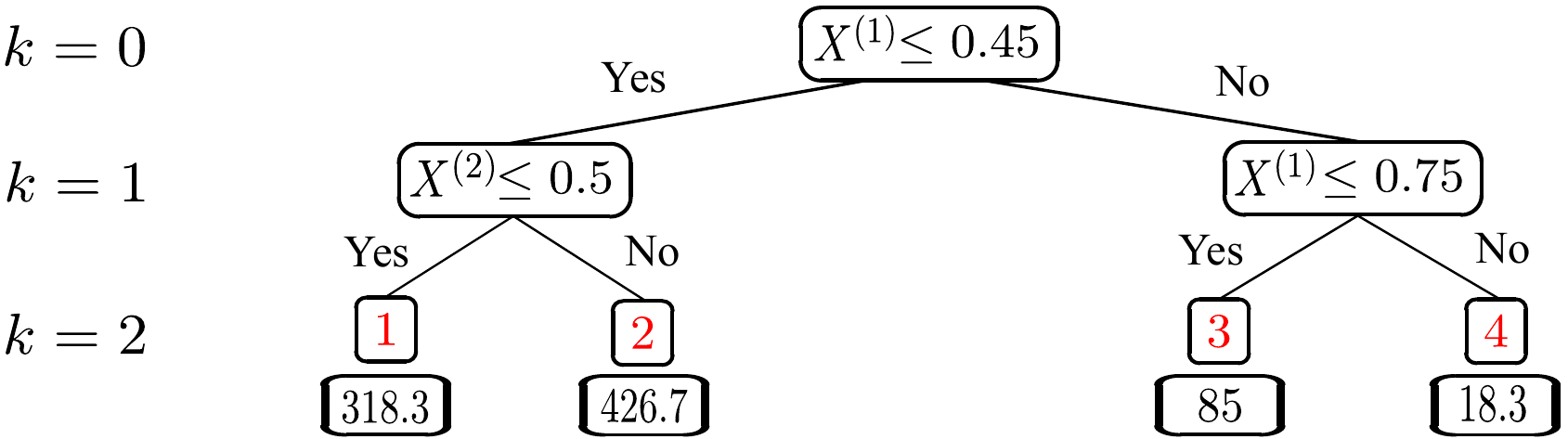}
	\end{minipage}
	\caption{A decision tree of depth $k=2$ in dimension $d=2$ (right) and the corresponding partition (left).}
	\label{figure1}
\end{figure}

\begin{algorithm}[htbp]
	\caption{CART construction.
	}
	\label{alg:split-cell} 
	\begin{algorithmic}[1]
		\STATE \textbf{Inputs:} data set $\mathcal{D}_n$, maximum number of observations in each leaf \texttt{nodesize}
		\STATE \textbf{Set} $\mathcal{A}_{\textrm{leaves}} = {\emptyset}$ and  $\mathcal{A}_{\textrm{inner nodes}} = \{ {[0,1]^d} \}$.
		\WHILE {$\mathcal{A}_{\textrm{inner nodes}} \neq \emptyset$} 
		\STATE Select the first element $A \in \mathcal{A}_{\textrm{inner nodes}}$
		\IF {$A$ contains more than \texttt{nodesize} observations}
		\STATE Select the best split minimizing the CART-split criterion defined below (see equation \eqref{definition_empirical_CART_criterion})
		\STATE Split the cell accordingly. 
		\STATE Concatenate the two resulting cell $A_L$ and $A_R$ at the end of $\mathcal{A}_{\textrm{inner nodes}}$.
		\STATE Remove $A$ from $\mathcal{A}_{\textrm{inner nodes}}$
		\ELSE 
		\STATE Concatenate $A$ at the end of $\mathcal{A}_{\textrm{leaves}}$.
		\STATE Remove $A$ from $\mathcal{A}_{\textrm{inner nodes}}$
		\ENDIF
		\ENDWHILE
		\STATE \textbf{Output:} The set of terminal leaves $\mathcal{A}_{\textrm{leaves}}$.
		%\STATE  For a query point $\bx$, the tree outputs the average $m_n(\bx)$ of the $Y_i$ falling into the same leaf as $\bx$.
%		\STATE coucou 
%		
%		\IF {lalala}
%		\STATE 
%		\ELSE
%		\STATE
%		\ENDIF
	\end{algorithmic}
\end{algorithm}

We define now the CART-split criterion. To this aim, we let $A \subset[0,1]^d $ be a generic cell and $N_n(A)$ be the number of data points falling into $A$. A split in $A$ is a pair $(j,z)$, where $j$ is a dimension in $\{1, \hdots, d\}$ and $z$ is the position of the cut along the $j$-th coordinate, within the limits of $A$. We let $\mathcal{C}_A$ be the set of all such possible cuts in $A$. Then, for any $(j,z) \in \mathcal{C}_A$, the CART-split criterion \citep[][]{BrFrOlSt84} takes the form
\begin{align}
L_{n,A}(j,z) = & \frac{1}{N_n(A)} \sum_{i=1}^n (Y_i - \bar{Y}_{A})^2\mathds{1}_{\bX_i \in A} \nonumber \\
&  - \frac{1}{N_n(A)} \sum_{i=1}^n (Y_i - \bar{Y}_{A_{L}} \mathds{1}_{\bX_i^{(j)}  < z} - \bar{Y}_{A_{R}} \mathds{1}_{\bX_i^{(j)} \geq z})^2 \mathds{1}_{\bX_i \in A},  \label{definition_empirical_CART_criterion}
\end{align}
where $A_L = \{ \bx \in A: \bx^{(j)} < z\}$, $A_R = \{ \bx \in A: \bx^{(j)} \geq z\}$, and $\bar{Y}_{A}$ (resp., $\bar{Y}_{A_{L}}$, $\bar{Y}_{A_{R}}$) is the average of the $Y_i$'s belonging to $A$ (resp., $A_{L}$, $A_{R}$), with the convention $0/0=0$. At each cell $A$, the best split $(j_{n,A}, z_{n,A})$ is finally selected by maximizing $L_{n,A}(j,z)$ over $\mathcal{C}_A$, that is
\begin{align}
(j_{n,A},z_{n,A}) \in \argmax\limits_{(j,z) \in \mathcal{C}_A } L_{n,A}(j,z).
\label{eq:best_empirical_splits}
\end{align}
To remove ties in the argmax, the best cut is always performed along the best cut direction $j_{n,A}$, at the middle of two consecutive data points. Once the tree has been built, the tree prediction $m_n(\bx)$ for a new query point $\bx$ is computed as the average of the $Y_i$ falling into the same leaf as $\bx$.

The CART-split criterion $L_{n,A}(j_{n,A},z_{n,A})$ can be rewritten as the difference in impurity between the cell $A$ and its two children, where the impurity is simply defined as the variance of the output in regression \citep[see also Chapter 4 in][]{BrFrOlSt84}. The impurity is thus related to the splitting criterion and can be used to assess variable importances via MDI. More precisely, the Mean Decrease in Impurity (MDI) for the variable $X^{(j)}$ computed via a tree $\tree$ is defined by 
\begin{align}
\empmdi_{\tree}(X^{(j)}) =  \sum_{\substack{A \in \mathcal{T}\\ j_{n,A} = j}}  p_{n,A} L_{n,A}(j_{n,A}, z_{n,A}), 
\label{eq:MDI_definition}
\end{align}
where the sum ranges over all cells $A$ in $\tree$ that are split along variable $j$ and $p_{n,A}$ is the fraction of observations falling into $A$. In other words, the MDI of $X^{(j)}$ computes the weighted decrease in impurity related to splits along the variable $X^{(j)}$.

\subsection{Theoretical CART construction and population MDI}
\label{subsec:theoretical_cart}

In order to study the (empirical) MDI defined in equation (\ref{eq:MDI_definition}), we first need to define and analyze the population version of MDI. First, we define a theoretical CART tree, as in Algorithm~\ref{alg:split-cell}, except for the splitting criterion which is replaced by its population version. Namely, for all cells $A$ and all splits $(j,z) \in \mathcal{C}_A$, the population version of the CART-split criterion is defined as
\begin{align}
L_A^{\star}(j,z) & = \V[ Y | \bX \in A ] - \P[ \bX^{(j)} < z\,|\, \bX \in A ]  ~\V [Y | \bX^{(j)} < z, \bX \in A ] \nonumber \\
& \qquad - \P[ \bX^{(j)} \geq z\,|\, \bX \in A ] ~\V [Y | \bX^{(j)} \geq z, \bX \in A ].
\label{eq:theoretical_cart_crit}
\end{align}
Therefore in each cell of a theoretical tree $\mathcal{T}^{\star}$, the best split $(j_A^{\star},z_A^{\star})$ is chosen by maximizing the population version of the CART-split criterion that is 
\begin{align*}
(j_A^{\star},z_A^{\star}) \in \argmax\limits_{(j,z) \in \mathcal{C}_A } L_A^{\star}(j,z).
\end{align*}
Of course, in practice, we cannot build a theoretical tree $\tree^{\star}$ since it relies on the true distribution $(\bX,Y)$ which is unknown. A theoretical tree is just an abstract mathematical object, for which we prove properties that will be  later extended to the (empirical) CART tree, our true object of interest. 

As for the empirical MDI defined above, the population MDI for the variable $X^{(j)}$ computed via the theoretical tree $\tree^{\star}$ is defined as
\begin{align*}
\mdi_{\tree^{\star}}(X^{(j)}) =  \sum_{\substack{A \in \tree^{\star}\\ j_{A}^{\star} = j}}  p^{\star}_{A} L_A^{\star}(j_{A}^{\star}, z_{A}^{\star}), 
\end{align*} 
where all empirical quantities have been replaced by their population version and the theoretical CART tree is used instead of the empirical CART tree defined in Section~\ref{subsec:emp_cart}. We also let $p^{\star}_{A} = \mathds{P}[\bX \in A]$.

\paragraph{Random forest} A random forest is nothing but a collection of several decision trees whose construction has been randomized. In Breiman's forest, CART procedure serves as base learner and the randomization is performed in two different ways: 
\begin{enumerate}
	\item Prior to the construction of each tree, a sample of the original data set is randomly drawn. Only this sample is used to build the tree. Sampling is typically done using bootstrap, by randomly drawing $n$ observations out of the original $n$ observations with replacement. 
	
	\item Additionally, for each cell, the best split is not selected along all possible variables. Instead, for each cell, a subsample of $\texttt{mtry}$ variables is randomly selected without replacement. The best split is selected only along these variables. By default, $\texttt{mtry} = d/3$ in regression. 
\end{enumerate}
Random forest prediction is then simply the average of the predictions of such randomized CART trees. Similarly, the MDI computed via a forest is nothing but the average of the MDI computed via each tree of the forest. 

In the sequel, we will use the  theoretical random forest (forest that aggregate theoretical CART trees) and the population MDI to derive results on the empirical random forest (forest that aggregate empirical CART trees; the one widely used in practice) and the empirical MDI. 

\section{Main theoretical result}
\label{sec:theoretical_results}

For any theoretical CART tree $\tree^{\star}$, and for any $k \in \mathds{N}$, we denote by $\tree_k^{\star}$ the truncation of $\tree^{\star}$ at level $k$, that is the subtree of $\tree^{\star}$ rooted at $[0,1]^d$, in which each leaf has been cut at most $k$ times. Let $A_{\tree_k^{\star}}(\bx)$ be the cell of the tree $\tree_k^{\star}$ containing $\bx$.
% We also let $A_{\tree^{\star}_{\infty}}(\bx) = \cap_{k=0}^{\infty} A_{\tree_k^{\star}}(\bx)$. 
%
For any function $f : [0,1]^d \to \mathds{R}$ and any cell $A \subset [0,1]^d$, let 
\begin{align*}
\Delta (f , A) = \mathds{V}[f(\bX) | \bX \in A]
\end{align*}
be the variance of the function $f$ in the cell $A$ with respect to the distribution of $\bX$. Proposition~\ref{thm:sumImportance} states that the population MDI computed via theoretical CART trees can be used to decompose the variance of the output, up to some residual term which depends on the risk of the theoretical tree estimate. 

\begin{proposition} \label{thm:sumImportance}
	Assume that $Y = m (\bX) + \varepsilon,$ where $\varepsilon$ is a noise satisfying $\mathds{E}[\varepsilon | \bX] = 0$ and  $\mathds{V}[\varepsilon|\bX] = \sigma^2$ almost surely, for some constant $\sigma^2$. 
	Consider a theoretical CART tree $\tree^{\star}$. Then, for all $k \geq 0$, 
	\begin{align}
	\mathds{V}[Y] = \sum_{j=1}^d \mdi_{\tree_k^{\star}}(X^{(j)})  + \mathds{E}_{\bX'} \big[ \mathds{V}[ Y | \bX \in A_{\tree_k^{\star}}(\bX') ]\big],
	\label{eq:imp_node_imp_leaves_pop}
	\end{align}
	where $\bX'$ is an i.i.d. copy of $\bX$. 
\end{proposition}	
	
Note that few assumptions are made in Proposition~\ref{thm:sumImportance}, which makes it very general. In addition, one can see that the sum of population MDI is close to the $R^2$ measure used to assess the quality of regression model. The population $R^2$ is defined as 
\begin{align*}
1 - \frac{\mathds{E}_{\bX'} \big[ \mathds{V}[ Y | \bX \in A_{\tree_k^{\star}}(\bX') ]\big]}{\mathds{V}[Y]} = \frac{\sum_{j=1}^d \mdi_{\tree_{k}^{\star}}(X^{(j)})}{\mathds{V}[Y]}.
\end{align*}
Hence, the sum of MDI divided by the variance of the output corresponds to the percentage of variance explained by the model, which is exactly the population $R^2$. Therefore, Proposition~\ref{thm:sumImportance} draws a clear connection between MDI and the very classical $R^2$ measure. 

We say that a theoretical tree $\tree^{\star}$ is consistent if $\lim_{k \to \infty} \Delta (m , A_{\tree_k^{\star}}(\bX')) = 0$. 
 If a theoretical tree is consistent, then its population $R^2$ tends to $\mathds{V}[m(\bX)]/\mathds{V}[Y]$ as shown in Corollary~\ref{cor:sumImportance} below.

\begin{corollary}
	\label{cor:sumImportance}
Grant assumptions of Proposition~\ref{thm:sumImportance}. Additionally, assume that $\|m\|_{\infty}<\infty$ and, almost surely, 
$\lim\limits_{k \to \infty} \Delta (m , A_{\tree_k^{\star}}(\bX')) = 0$. Then, for all $\gamma>0$, there exists $K$ such that, for all $k >K$, 
\begin{align*}
\Big| \sum_{j=1}^d \mdi_{\tree_{k}^{\star}}(X^{(j)}) - \mathds{V}[m(\bX)]\Big| \leq \gamma
\end{align*}
\end{corollary}

Consistency (and, in this case, tree consistency) is a prerequisite when dealing with algorithm intepretability. Indeed, it is hopeless to think that we can leverage information about the true data distribution from an inconsistent algorithm, intrinsically incapable of modelling these data. Therefore, we assume in this section that theoretical trees are consistent and study afterwards variable importances produced by such algorithm. This strategy has been also adopted in \citet{ramosaj2019asymptotic}, who assume the consistency of the (infinite) empirical random forests \citep[Assumption A5, page 6; see also Assumption A2, page 9 in][]{benard2021mda}.
%Note that the consistency assumption in Corollary~\ref{cor:sumImportance} is about the theoretical CART partition and not the empirical CART partition. 
In the next sections, we will be able to prove tree consistency for specific regression models, allowing us to remove the generic consistency assumption in our results, as the one in Corollary~\ref{cor:sumImportance}. 

%\es{Commenter l'hypothèse des arbres, dire qu'on va réutiliser Delta dans la suite}
%The assumption $\lim_{k \to \infty} \Delta (m , A_{\tree_k^{\star}}(\bX')) = 0$ states that the theoretical tree $\tree_k^{\star}$ is consistent, with respect to the true data distribution (that is, replacing empirical mean by expectation in the tree leaves). Therefore, 

Corollary~\ref{cor:sumImportance} states that if the theoretical tree $\tree_k^{\star}$ is consistent, then the sum of the MDI of all variables tends to the total amount of information of the model, that is $\mathds{V}[m(\bX)]$. This gives a nice interpretation of MDI as a way to decompose the total variance $\mathds{V}[m(\bX)]$ across variables. 
\cite{LoWeSuGe13} prove a similar result when all features are categorical. This result was later extended by  \cite{sutera2016context} in the case of context-dependent features. In their case, trees are finite since there exists only a finite number of variables with a finite number of categories. The major difference with our analysis is that trees we study can be grown to arbitrary depth, since we consider continuous input variables, along which an infinite number of splits can be performed. We thus need to control the error of a tree stopped at some level $k$, which is exactly the right-hand term in equation~\eqref{eq:imp_node_imp_leaves_pop}.

According to Corollary~\ref{cor:sumImportance}, the sum of the MDI is the same for all consistent theoretical trees: even if there may exist many theoretical trees, the sum of MDI computed for each theoretical tree tends to the same value, regardless of the considered tree. Therefore, all consistent theoretical tree produce the same asymptotic value for the sum of population MDI. 

In practice, one cannot build a theoretical CART tree or compute the population MDI. Proposition~\ref{thm:sumImportance_emp} below is the extension of Proposition~\ref{thm:sumImportance} for the empirical CART procedure. Let $\tree_n$ be the (empirical) CART tree built with data set $\mathcal{D}_n$ and let, for all $k$, $\tree_{n,k}$ be the truncation of $\tree_{n}$ at level $k$. 
Denote by $\widehat{\mathds{V}[Y]}$ the empirical variance of $Y$ computed on the data set $\mathcal{D}_n$. Define, for any function $f: [0,1]^d \mapsto \mathds{R}$, the empirical quadratic risk via
\begin{align*}
R_n(f) = \frac{1}{n} \sum_{i=1}^n (Y_i - f(X_i))^2.
\end{align*}

\begin{proposition}[Empirical version of Proposition~\ref{thm:sumImportance}]
\label{thm:sumImportance_emp}
Let $\tree_n$ be the CART tree, based on the data set $\mathcal{D}_n$. Then, 
\begin{align}
\widehat{\mathds{V}[Y]}  =  \sum_{j=1}^d \empmdi_{\tree_n}(X^{(j)}) + R_n(f_{\tree_n}), \label{eq:imp_node_imp_leaves_emp}
\end{align}
where $f_{\tree_n}$ is the estimate associated to $\tree_n$, as defined in Algorithm~\ref{alg:split-cell}.
\end{proposition}

Note that equations \eqref{eq:imp_node_imp_leaves_pop} and \eqref{eq:imp_node_imp_leaves_emp} in Proposition~\ref{thm:sumImportance} and Proposition~\ref{thm:sumImportance_emp} are valid for general tree construction. The main argument of the proof is simply a telescopic sum of the MDI which links the variance of $Y$ in the root node of the tree to the variance of $Y$ in each terminal node. These equalities hold for general impurity measures by providing a relation between root impurity and leaves impurity. 
As in Proposition~\ref{thm:sumImportance}, Proposition~\ref{thm:sumImportance_emp} proves that the $R^2$ of a tree can be written as 
\begin{align}
\frac{\sum_{j=1}^d \empmdi_{\tree_n}(X^{(j)})}{\widehat{\mathds{V}[Y]}},
\end{align}
which allows us to see the sum of MDI as the percentage of variance explained by the model. This is particularly informative about the quality of the tree modelling when the depth $k$ of the tree is fixed. 

Trees are likely to overfit when they are fully grown. Indeed, the risk of trees which contain only one observation per leaf is exactly zero. Hence, according to equation~\eqref{eq:imp_node_imp_leaves_emp}, we have
\begin{align*}
\widehat{\mathds{V}[Y]} = \sum_{j=1}^d \empmdi_{\tree_n}(X^{(j)}) 
\end{align*}
and the $R^2$ of such tree is equal to one. Such $R^2$ does not mean that trees have good generalization error but that they have enough approximation capacity to overfit the data. Additionally, for a fully grown tree $\tree_n$, we have
\begin{align*}
\lim\limits_{n \to \infty} \sum_{j=1}^d \empmdi_{\tree_n}(X^{(j)}) = \mathds{V}[m(\bX)] + \sigma^2.
\end{align*}
For fully grown trees, the sum of MDI contains not only all available information $\mathds{V}[m(\bX)]$ but also the noise in the data. This implies that the MDI of some variables are higher than expected due to the noise in the data. 
The noise, when having a larger variance than the regression function, can induce a very important bias in the MDI by overestimating the importance of some variables. The bias of MDI has already been empirically studied. To circumvent the overfitting problem of MDI computations (using a single data set to build the trees and estimate the variable importances),  \cite{zhou2019unbiased} use a test set to compute MDI that are less biased. A quantification of the MDI bias has been proposed recently in \cite{li2019debiased} in which they show that deep trees tend to exhibit larger MDI bias for irrelevant variables. Their analysis provides finite-sample bounds for MDI of irrelevant variables whereas our analysis is asymptotic for both relevant and irrelevant variables.

The above discussion stresses out that, when interested by interpreting a decision tree, one must never use the MDI measures output by a fully grown tree. However, if the depth of the tree is fixed and large enough, the MDI output by the tree provides a good estimation of the population MDI as shown in Theorem~\ref{thm:sumImportance_emp2}.

\begin{model}
	\label{ass:uniform_input}
	The regression model satisfies $Y = m(\bX) + \varepsilon$ where $m$ is continuous, $\bX$ admits a density bounded from above and below by some positive constants and $\varepsilon$ is an independent Gaussian noise of variance $\sigma^2$. 
\end{model}

\begin{theorem}
	\label{thm:sumImportance_emp2}
Assume Model~\ref{ass:uniform_input} holds and that for all theoretical CART tree $\tree^{\star}$, almost surely, 
$$\lim_{k \to \infty} \Delta (m , A_{\tree_k^{\star}}(\bX)) = 0.$$ Let $\tree_n$ be the CART tree based on the data set $\mathcal{D}_n$.
Then, for all $\gamma>0, \rho \in (0,1]$, there exists $K \in \mathds{N}^{\star}$ such that, for all $k>K$, for all $n$ large enough, with probability at least $1-\rho$, 
\begin{align*}
\Big| \sum_{j=1}^d \empmdi_{\tree_{n,k}}(X^{(j)}) - \mathds{V}[m(\bX)]   \Big| \leq \gamma.
\end{align*}
\end{theorem}

As for Corollary~\ref{cor:sumImportance}, Theorem~\ref{thm:sumImportance_emp2} relies on the consistency of theoretical trees, which is essential to obtain positive results on tree interpretability. It is worth noticing that Theorem~\ref{thm:sumImportance_emp2} is not a straightforward extension of Corollary~\ref{cor:sumImportance}. The proof is based on the uniform consistency of theoretical trees and combines several recent results on tree partitions to establish the consistency of empirical CART tree. 

It is not possible in general to make explicit the contribution of each MDI to the sum. In other words, it is not easy to find an explicit expression of the MDI of each variable. This is due to interactions and correlation between inputs, which make difficult to distinguish the impact of each variable. 
Nevertheless, when the regression model can be decomposed into a sum of two independent terms, we can obtain more precise information on MDI. To this aim, let, for any set $\mathcal{J}= \{j_1, \hdots, j_J\} \subset \{1, \hdots, d\}$, $\bX^{(\mathcal{J})} = (X^{(j_1)}, \hdots, X^{(j_J)})$.

\begin{model}
	\label{ass:decomp_m_sum}
	There exists $\mathcal{J} \subset  \{1, \hdots , d \}$ such that $\bX^{(\mathcal{J})}$ and $\bX^{(\mathcal{J}^c)}$ are independent and 
	\begin{align}
	Y = m_{\mathcal{J}}(\bX^{(\mathcal{J})}) +  m_{\mathcal{J}^c}(\bX^{(\mathcal{J}^c)}) + \varepsilon, \label{eq:partial_add_model}
	\end{align}
	where $m_{\mathcal{J}}$ and $m_{\mathcal{J}^c}$ are continuous functions, $\varepsilon$ is a Gaussian noise $\mathcal{N}(0, \sigma^2)$ independent of $\bX$ and $\bX$ admits a density bounded from above and below by some positive constants. 
\end{model}

\begin{lemme}
	\label{Lem:decompt_th_3}
Assume Model~\ref{ass:decomp_m_sum} holds. Then, for all $j \in \mathcal{J}$, the criterion $L^{\star}_{A}(j,s) $ does not depend on $m_{\mathcal{J}^c}$ and is equal to $L^{\star}_{A^{(\mathcal{J})}}(j,s) $. Besides, any split $j \in \mathcal{J}$ leaves the variance of $m_{\mathcal{J}^c}$ unchanged.
\end{lemme}

Lemma~\ref{Lem:decompt_th_3} will be used to prove Proposition~\ref{prop:th_decomp_th3} below but is informative on its own. It states that if the regression model can be decomposed as in Model~\ref{ass:decomp_m_sum}, the best split along coordinates in $\mathcal{J}$ does not depend on $m_{\mathcal{J}^c}$ or on the ranges of coordinates in $\mathcal{J}^c$.

\begin{proposition}
	\label{prop:th_decomp_th3}
Assume Model~\ref{ass:decomp_m_sum} holds. Let $\tree^{\star}_{\mathcal{J}}$ be a tree built on distribution $(\bX^{(\mathcal{J})},$ $ m_{\mathcal{J}}(\bX^{(\mathcal{J})}))$  and let $\tree^{\star}_{\mathcal{J},k}$ be the truncation of the tree at level $k$. Similarly, let $\tree^{\star}_{\mathcal{J}^c}$ be a tree built on distribution  $(\bX^{(\mathcal{J}^c)}, m_{\mathcal{J}^c}(\bX^{(\mathcal{J}^c)}))$  and let $\tree^{\star}_{\mathcal{J}^c, k}$ be the truncation of the tree at level $k$.  If for any tree $\tree^{\star}_{\mathcal{J}}$ and $\tree^{\star}_{\mathcal{J}^c}$, almost surely, 
	\begin{align*}
	\lim\limits_{k \to \infty} \Delta (m_{\mathcal{J}} , A_{\tree_{\mathcal{J},k}^{\star}}(\bX^{(\mathcal{J})})) = 0, \quad \textrm{and} \quad \lim\limits_{k \to \infty} \Delta (m_{\mathcal{J}^c} , A_{\tree_{\mathcal{J}^c,k}^{\star}}(\bX^{(\mathcal{J}^c)})) = 0
	\end{align*}
	then, for any tree $\tree^{\star}$ built on the original distribution $(\bX,Y)$, almost surely,
	\begin{align*}
	\lim\limits_{k \to \infty} \Delta (m , A_{\tree_k^{\star}}(\bX)) = 0.
	\end{align*}
\end{proposition}

Proposition~\ref{prop:th_decomp_th3} states that the consistency of any theoretical trees results from the consistency of all theoretical trees built on subsets ($\mathcal{J}$ and $\mathcal{J}^c$) of the original set of variables, under Model~\ref{ass:decomp_m_sum}. This is of particular significance for proving the consistency of theoretical tree for a new model which can be rewritten as a sum of independent regression functions: we can then deduce the tree consistency from existing consistency results for each regression function. Note that there are no structural assumptions on the functions $m_{\mathcal{J}}$ and $m_{\mathcal{J}^c}$ apart from the continuity. In particular, the class of functions in Model~\ref{ass:decomp_m_sum} is wider than additive functions. Note also that importance measures via random forests have been designed for groups of variables. For instance, \citet{gregorutti2015grouped} extended the MDA to groups of variables and studied it theoretically in an additive regression framework. More recently, a tree algorithm tailored for groups of variables have been proposed by \citet{poterie2019classification}.

%Theorem~\ref{thm:mdi_interaction} below proves that the sum of the MDI inside groups of variables that $(i)$ do not interact between each other and $(ii)$ are independent of each other is well-defined and equals the variance of the corresponding component of the regression function.  

Theorem~\ref{thm:mdi_interaction} below establishes the limiting value of the sum of MDI in subgroups $\mathcal{J}$ and $\mathcal{J}^c$, both for the population version and the empirical version of the MDI, as soon as all theoretical trees are consistent. 
 
\begin{theorem}\label{thm:mdi_interaction}
	Under the assumptions of Proposition~\ref{prop:th_decomp_th3}, for all $\gamma >0$, there exists $K$ such that, for all $k>K$,
		\begin{align*}
		\Big| \sum_{j \in \mathcal{J}} \mdi_{\tree^{\star}_k}(X^{(j)}) - \mathds{V}[m_{\mathcal{J}}(\bX^{(\mathcal{J})})]\Big| \leq \gamma,
		\end{align*}
		and
		\begin{align*}
		\Big| \sum_{j \in \mathcal{J}^c} \mdi_{\tree^{\star}_k}(X^{(j)}) - \mathds{V}[m_{\mathcal{J}^c}(\bX^{(\mathcal{J}^c)})] \Big| \leq \gamma.
		\end{align*}
		
		Besides, for all $\gamma>0, \rho \in (0,1]$, there exists $K$ such that, for all $k>K$, for all $n$ large enough, with probability at least $1 - \rho$,
		\begin{align*}
		\Big| \sum_{j \in \mathcal{J}} \empmdi_{\tree_{n,K}}(X^{(j)}) - \mathds{V}[m_{\mathcal{J}}(\bX^{(\mathcal{J})})]\Big| \leq \gamma,
		\end{align*}
		and
		\begin{align*}
		\Big| \sum_{j \in \mathcal{J}^c} \empmdi_{\tree_{n,K}}(X^{(j)}) - \mathds{V}[m_{\mathcal{J}^c}(\bX^{(\mathcal{J}^c)})] \Big| \leq \gamma.
		\end{align*}	
\end{theorem}

According to Theorem~\ref{thm:mdi_interaction}, the limiting value of the MDI for subgroup $\mathcal{J}$ (resp. subgroup $\mathcal{J}^c$) is equal to the variance of the corresponding component of the regression function, that is $\mathds{V}[m_{\mathcal{J}}(\bX^{(\mathcal{J})})]$ (resp. $\mathds{V}[m_{\mathcal{J}^c}(\bX^{(\mathcal{J}^c)})]$). We will use Theorem~\ref{thm:mdi_interaction} in the Section~\ref{sec:additive_models} to derive explicit expressions for the MDI of each variable, assuming a specific form of the regression function. 

\section{Additive Models}
\label{sec:additive_models}

To pursue our analysis, we consider in this section a specific type of regression models: additive regression functions with independents inputs. 

\begin{model}[Additive model]
	\label{ass:gaussian_noise}
	The regression model writes
	\begin{align*}
	Y = \sum_{j=1}^d m_j(X^{(j)}) + \varepsilon,
	\end{align*}
	where each $m_j$ is continuous; $\varepsilon$ is a  Gaussian noise $\mathcal{N}(0,\sigma^2)$, independent of $\bX$; and $ \bX \sim \mathcal{U}([0,1]^d)$. 
\end{model}

Additive regression models were popularized by \citet{St85} \citep[see also ][]{hastie2017generalized}. They were used as a generic framework for studying random forests in some previous works \citet{klusowski2020sparse, ScBiVe15}, notably due to their simple structure and their absence of interactions. Indeed, due to the model additivity, there is no interaction between variables: the contribution of each variable for predicting the output is reduced to its marginal contribution. Since we also assume that features are independent, the information contained in one variable cannot be inferred by the values of other variables. 

By considering a model with no interaction and with independent inputs, we know that the contribution of a variable has an intrinsic definition which does not depend on other variables that are used to build the model.  In Model~\ref{ass:gaussian_noise}, the explained variance of the model $\mathds{V}[m(\bX)]$ takes the form 
\begin{align*}
\mathds{V}[m(\bX)] = \sum_{j=1}^d \mathds{V}[m_j(X^{(j)})],
\end{align*} 
which holds only because of independent inputs \textit{and} the absence of interactions. The variance explained by the $j$th variable is defined unambiguously, and independently of which variables are considered in the model, as $\mathds{V}[m_j(X^{(j)})]$. Therefore any importance measure for $X^{(j)}$ which aims at decomposing the explained variance must output $\mathds{V}[m_j(X^{(j)})]$. It turns out that the MDI computed via CART trees converges to this quantity.

\begin{theorem}[Additive model]
\label{thm:additive_models}
Assume that Model~\ref{ass:gaussian_noise} holds. Let $\tree^{\star}$ be any theoretical CART tree and $\tree_n$ be the empirical CART tree. Then, for all $\gamma >0$, there exists $K$ such that, for all $k>K$, for all $j$,
\begin{align*}
\Big| \mdi_{\tree^{\star}_k}(X^{(j)}) - \mathds{V}[m_j(X^{(j)})] \Big| \leq \gamma,
\end{align*}
Moreover, for all $\gamma>0, \rho \in (0,1]$, there exists $K$ such that, for all $k>K$, for all $n$ large enough, with probability at least $1 - \rho$, for all $j$,
\begin{align*}
\Big| \empmdi_{\tree_{n,k}}(X^{(j)}) - \mathds{V}[m_j(X^{(j)})] \Big| \leq \gamma.
\end{align*}
\end{theorem}

Since the MDI computed via random forests is nothing but the average of MDI computed by trees, Theorem~\ref{thm:additive_models} also holds for MDI computed via random forests. To the best of our knowledge, Theorem~\ref{thm:additive_models} is the first result highlighting that empirical MDI computed with CART procedure converge to reasonable values that can be used for variable selection, in the framework of Model~\ref{ass:gaussian_noise}. 

Contrary to Theorem~\ref{thm:mdi_interaction},  Theorem~\ref{thm:additive_models} does not assume the consistency of the theoretical tree. Indeed, in the context of additive models, one can directly take advantage of \cite{ScBiVe15} to prove the consistency of theoretical CART trees. 

Surprisingly, the population version of MDI has the same expression as the population version of MDA (up to factor 2), as stated in \cite{gregorutti2017correlation}. Thus, in the particular context of additive models with independent features, both MDI and MDA target the same quantity, which is the natural way to decompose the total variance. In this context, MDI and MDA can then be employed to rank features and select variables based on the importance values they produce. 

Note that MDI is computed with truncated trees, i.e. trees that contain a large number of observations in each node. This is mandatory to ensure that the variance of the output in the resulting cell is correctly estimated. As mentioned before, considering fully grown trees would result in positively-biased MDI, which can lead to unintended consequences for variable selection. We therefore stress that MDI must not be computed using fully grown trees. 

Theorem~\ref{thm:additive_models} proves that MDI computed with any CART tree are asymptotically identical: the MDI are consistent across trees.  The only interest to use many trees to compute the MDI instead of a single one relies on the variance reduction property of ensemble methods, which allows us to reduce the variance of the MDI estimate when using many trees.

\paragraph{Remark: invariance under monotonous transformations} Tree-based methods are known to be invariant by strictly monotonous transformations of each input. Letting $f_j$ being any monotonous transformation applied to variable $X^{(j)}$, the new regression model writes
\begin{align*}
Y = \sum_{j=1}^d m_j \circ f_j^{-1} (f_j(X^{(j)})) + \varepsilon,
\end{align*}
where $f_j(X^{(j)})$ are the new input variables. Assuming that Theorem~\ref{thm:additive_models} can be applied in the more general setting where $X^{(j)} \sim f_j(\mathcal{U}([0,1]))$, the MDI for this new problem satisfies
\begin{align*}
\lim\limits_{k \to \infty} \mdi_{\tree_k^{\star}}(X^{(j)}) = \mathds{V}[m_j \circ f_j^{-1} (f_j(X^{(j)}))] = \mathds{V}[m_j(X^{(j)})].
\end{align*}
Therefore the asymptotic MDI is invariant by monotonic transformation, assuming that Theorem~\ref{thm:additive_models} holds for any marginal distribution of $\bX$ (with independent components).

\medskip 

The most famous instance of additive regression models is the well-studied linear model \citep[see, e.g.,][]{rencher2008linear, searle2016linear}. 
\begin{model}[Linear model]
	\label{ass:lin_models}
	The regression model writes
	\begin{align*}
	Y = \sum_{j=1}^d \alpha_j X^{(j)} + \varepsilon,
	\end{align*}
	where, for all $j$, $\alpha_j \in \mathds{R}$; $\varepsilon$ is a  Gaussian noise $\mathcal{N}(0,\sigma^2)$, independent of $\bX$; and $ \bX \sim \mathcal{U}([0,1]^d)$. 
\end{model}
As for  Model~\ref{ass:gaussian_noise}, there is only one way to decompose the explained variance for linear models with independent input, as stated by \cite{nathans2012interpreting} and \cite{gromping2015variable} (and the references therein), which is $\alpha_j^2 \mathds{V}[X^{(j)}]$. Theorem~\ref{thm:lin_reg} below proves that MDI converge to these quantities and provides information about the theoretical tree structure in Model~\ref{ass:lin_models}.

\begin{theorem}[Linear model]
	\label{thm:lin_reg}
	Assume that Model~\ref{ass:lin_models} holds. For any cell $A =\prod_{\ell=1}^d [a_{\ell}, b_{\ell}] \subset [0,1]$, the coordinate which maximizes the population CART-split criterion (\ref{eq:theoretical_cart_crit}) satisfies 
	\begin{align*}
		j^{\star} \in \argmax_{\ell \in \{1, \hdots , d\}} \alpha_{\ell}^2 (b_{\ell}-a_{\ell})^2\, 
	\end{align*}
	and the splitting position is the center of the cell; the associated variance reduction is 
	\begin{align*}
		\frac{\alpha_{j^{\star}}^2}{4} \Big[ \frac{b_{j^{\star}}-a_{j^{\star}}}{2}\Big]^2.
	\end{align*}
	 Let $\tree^{\star}$ be any theoretical CART tree and $\tree_n$ be the empirical CART tree. Then, for all $\gamma >0$, there exists $K$ such that, for all $k>K$ and for all $j$,
	 \begin{align*}
	 \Big| \mdi_{\tree_k^{\star}}(X^{(j)}) - \frac{\alpha_j^2}{12} \Big| \leq \gamma.
	 \end{align*}
	Moreover, for all $\gamma>0, \rho \in (0,1]$, there exists $K$ such that, for all $k>K$, for all $n$ large enough, with probability at least $1 - \rho$, for all $j$,
	\begin{align*}
	\Big| \empmdi_{\tree_{n,k}}(X^{(j)}) - \frac{\alpha_j^2}{12} \Big| \leq \gamma.
	\end{align*}
\end{theorem}

Theorem~\ref{thm:lin_reg} establishes that, for linear models with independent inputs, the MDI computed with CART trees whose depth is limited is exactly the quantity of interest, which depends only on the magnitude of $\alpha_j^2$, since all variables have the same variance.  
Theorem~\ref{thm:lin_reg} is a direct consequence of Theorem~\ref{thm:additive_models}. However, for linear models, one can compute explicitly the best splitting location and the associated variance decreasing, which is stated in Theorem~\ref{thm:lin_reg}. Being able to compute the variance decreasing for each split allows us to compute the MDI directly, without using Theorem~\ref{thm:additive_models}. Note that the splitting location and variance reduction has been first stated by \cite{Bi12} and then by \cite{Is13}.

\begin{remark}
All our results focus on the asymptotic behaviour of the MDI computed with CART. A natural next step in our analysis would be to derive rates of convergence for MDI, in order to quantify the accuracy of a variable ranking based on MDI. Unfortunately, we are limited by the (asymptotic) tree consistency assumption (see assumption in Corollary~\ref{cor:sumImportance}). If we had access to the rate of consistency of a tree, we could leverage on that to establish upper bounds on the rate of consistency of MDI. Although interesting, this line of work seems difficult. However, some generic arguments based on partition shapes have been used to establish rate of consistency for groups of MDI \citep[MDI of irrelevant variables, see Theorem~1 in][]{li2019debiased}. Interesting finite-sample lower bounds on MDI can also be found in \citet{klusowski2019analyzing}.
\end{remark}

In this section, we studied additive (linear) models and proved that, for these models with independent inputs, MDI computed with truncated CART is consistent. This result is based on CART consistency proved in \citet{ScBiVe15} for additive models with independent inputs. One can naturally wonder what happens if the additivity assumption (or the independence assumption) is removed. In these cases, the general approach, developed in this section, is unfortunately not valid. The next two sections focus on two specific settings, and show that, when one of these two assumptions is removed, the MDI is ill-defined.

\section{A model with interactions}
\label{sec:interactions}

In the absence of interaction and dependence between inputs, the MDI is a measure of variable importance which is independent of the set of variables included in the model, as stated in Theorem~\ref{thm:additive_models}. However, this is not true as soon as we consider classes of models that present interactions or dependence. In this section, we study the influence of interactions via the following model. 

\begin{model}[Multiplicative model]
	\label{ass:mult_models}
	Let $\alpha \in \mathds{R}$. The regression model writes
	\begin{align*}
	Y = 2^d \alpha \prod_{j=1}^d X^{(j)}+ \varepsilon.
	\end{align*}
	where  $\alpha \in \mathds{R}$; $\varepsilon$ is a  Gaussian noise $\mathcal{N}(0,\sigma^2)$, independent of $\bX$; and $ \bX \sim \mathcal{U}([0,1]^d)$. 
\end{model}

\begin{theorem}[Multiplicative model]
	\label{thm:dim2_product_mdibis}
	Assume that Model~\ref{ass:mult_models} holds. For any cell $A =\prod_{\ell=1}^d [a_{\ell}, b_{\ell}] \subset [0,1]$, the coordinate which maximizes the population CART-split criterion (\ref{eq:theoretical_cart_crit}) satisfies 
	\begin{align*}
	j^{\star} \in \argmax_{\ell \in \{1, \hdots , d\}}  \frac{b_{\ell} - a_{\ell} }{a_{\ell} + b_{\ell}},
	\end{align*} 
	and the splitting position is the center of the cell; the associated variance reduction is 
	\begin{align*}
	\frac{\alpha^2}{4} (b_{j^{\star}} - a_{j^{\star}})^2 \prod_{\ell \neq j^{\star}} (a_{\ell} + b_{\ell})^2
	\end{align*}
	Let $\tree^{\star}$ be any theoretical CART tree and $\tree_n$ be the empirical CART tree. Then, for all $\gamma >0$, there exists $K$ such that, for all $k >K$,
	\begin{align*}
	\Big| \sum_{j=1}^d  \mdi_{\tree_k^{\star}}(X^{(j)}) - \alpha^2 \Big(\Big(\frac{4}{3}\Big)^d - 1 \Big) \Big| \leq \gamma.
	\end{align*}
	Moreover, for all $\gamma>0, \rho \in (0,1]$, there exists $K$ such that, for all $k>K$, for all $n$ large enough, with probability at least $1 - \rho$, 
	\begin{align*}
	\Big| \sum_{j=1}^d  \empmdi_{\tree_{n,k}}(X^{(j)}) - \alpha^2 \Big(\Big(\frac{4}{3}\Big)^d - 1 \Big)  \Big| \leq \gamma.
	\end{align*}
\end{theorem}

Theorem~\ref{thm:dim2_product_mdibis} gives the explicit splitting position for a model with interactions. Deriving splitting positions allows us to prove that the theoretical tree is consistent which, in turns, proves that the sum of MDI converges to the explained variance $\mathds{V}[m(\bX)]$, according to Theorem~\ref{thm:sumImportance_emp2}. 

We are also interested in obtaining the exact MDI expression for each variable, as established for additive models. However, Theorem~\ref{thm:mdi_interaction} no longer applies since the regression model cannot be decomposed into independent additive components. It turns out to be impossible to derive explicit expression for each MDI in this model. 
To see this, note that there exist many theoretical trees in Model~\ref{ass:mult_models}. Two of them are displayed in Figure~\ref{figure2}. They result in the same partition but the first split can be made either along $X^{(1)}$ (Figure~\ref{figure2}, left)  or $X^{(2)}$ (Figure~\ref{figure2}, right).
Surprisingly, the variable importance computed with these trees is larger for the variable that is split in the second step. This exemplifies that splits occurring in the first levels of trees do not lead to the largest decrease in variance, which in turn, shows that variable importances based on the  frequency of splitting along a variable in the first level(s) of a tree are inaccurate. Nonetheless, this holds only when interactions are present: in the case of the additive models defined in Section~\ref{sec:additive_models}, the decrease in variance is stronger in the first tree levels. 
A direct consequence of this fact is that  two different theoretical tree can output two different variable importance, as shown by Lemma~\ref{lem:diff_imp} below.

\begin{figure}[h!!]
	\begin{minipage}[c]{.5\linewidth}
		\begin{center}
		\includegraphics[scale=0.5]{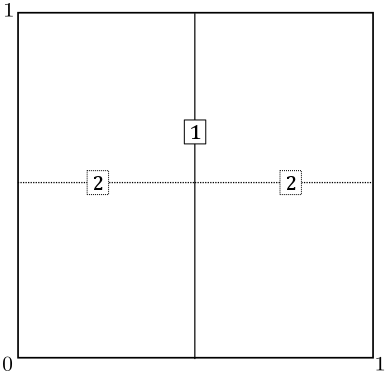}
		\end{center}
	\end{minipage} 
	\hspace{-0.5cm}
	\begin{minipage}[c]{.5\linewidth}
		\begin{center}
		\includegraphics[scale=0.5]{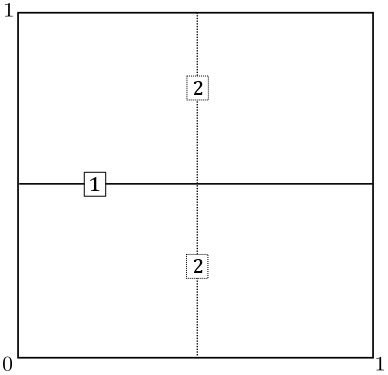}
		\end{center}
	\end{minipage}
	\caption{Two theoretical tree partitions of level $k=2$: the first split is performed along $X^{(1)}$ (resp. $X^{(2)}$) for the tree on the left (resp. on the right).}
	\label{figure2}
\end{figure}

\begin{lemme} \label{lem:diff_imp}
Assume that Model~\ref{ass:mult_models} holds. Then, there exists two  theoretical trees $\mathcal{T}_{1}$ and $\mathcal{T}_{2}$ such that
\begin{align*}
\lim\limits_{k \to \infty} \Big( \mdi_{\mathcal{T}_{2,k}}(X^{(1)}) - \mdi_{\mathcal{T}_{1,k}}(X^{(1)}) \big) = \alpha^2/16.
\end{align*} 
\end{lemme}
According to Lemma~\ref{lem:diff_imp}, there exist many different theoretical trees for Model~\ref{ass:mult_models} and the MDI may vary when computed with different theoretical trees. This is a major difference with additive models for which each theoretical tree asymptotically output the same MDI. Since MDI are usually used to rank and select variables, the fact that each tree can output a different MDI, and therefore a different variable ranking, is a major drawback for employing this measure in presence of interactions. One way to circumvent this issue could be to compute the MDI via a random forest: the randomization of splitting variables yields different trees and the aggregation of MDI across trees provides a mean effect of the variable importance.  For example, in Model~\ref{ass:mult_models}, one can hope to obtain importance measure that are equal since Model~\ref{ass:mult_models} is symmetric in all variables. This is impossible with only one tree but is a reachable goal by computing the mean of MDI across many trees as done in random forests. More generally, for simple regression problems, a single tree may be sufficient to predict accurately but many shallow trees are needed to obtain accurate MDI values.

\section{A model with dependent inputs}
\label{sec:corr}

Our previous results relies on the (often unrealistic) assumption that input variables are uniformly distributed over $[0,1]^d$, and thus independent. To gain insights on the impact of input dependence on MDI, we study a very specific model in which input features are not independent. We will prove that even in this very simple model, MDI has some severe drawbacks that need to be addressed carefully. 
Accordingly, we consider the simple case where the input vector $\bX = (X^{(1)}, X^{(2)})$ is of dimension two and has a distribution, parametrized by $\beta \in \mathds{N}$, defined as 
\begin{align}
\bX \sim \mathcal{U}\left( \cup_{j=0}^{2^{\beta}-1} \Big[ \frac{j}{2^{\beta}}, \frac{j+1}{2^{\beta}}\Big)^2 \right) = \unifbeta.
\label{ass:distribution_correlation}
\end{align}
The distribution of $\bX$ is uniform on $2^{\beta}$ squares on the diagonal. Examples of such distribution are displayed in Figure~\ref{figure3}.  
\begin{figure}[h!]
	\begin{minipage}[c]{.5\linewidth}
		\begin{center}
		\includegraphics[scale=0.5]{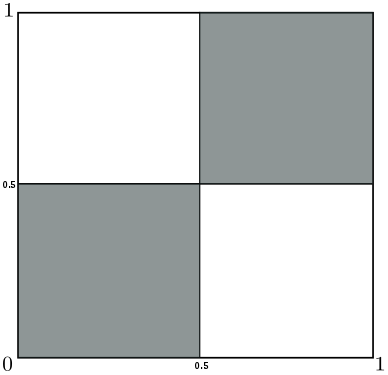}
		\end{center}
		
	\end{minipage} 
	\hspace{-0.5cm}
	\begin{minipage}[c]{.5\linewidth}
		\begin{center}
		\includegraphics[scale=0.5]{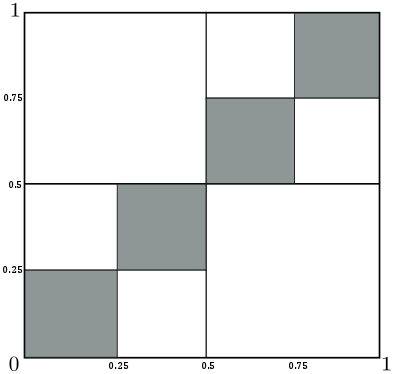}
		\end{center}	
	\end{minipage}
	\caption{Illustration of distributions for $\bX$, with parameter $\beta=1$ (left) and $\beta=2$ (right).}
	\label{figure3}
\end{figure}
For such distribution, the correlation between $X^{(1)}$ and $X^{(2)}$ is given by Lemma~\ref{lemma:correlationX1X2}.
\begin{lemme} \label{lemma:correlationX1X2}
	Let $\bX = (X^{(1)}, X^{(2)}) \sim \unifbeta$ as defined in  (\ref{ass:distribution_correlation}). Then 
	\begin{align*}
	\corr(X^{(1)}, X^{(2)}) =  1 - \frac{1}{4^{\beta}}.
	\end{align*}
\end{lemme} 
For $\beta=0$, the distribution of $\bX$ is uniform over $[0,1]^2$ and accordingly, the correlation is null between the two components of $\bX$, as stated in Lemma~\ref{lemma:correlationX1X2}. When $\beta$ (i.e. the number of squares) increases, that is the size of each square decreases, the distribution concentrates on the line $X^{(1)} = X^{(2)}$. Therefore, when $\beta$ tends to infinity, the correlation should tend to $1$, which is proved in  Lemma~\ref{lemma:correlationX1X2}.

The distribution defined in (\ref{ass:distribution_correlation}) is rather unusual, and one can wonder why not considering Gaussian distributions instead. This choice is related to the thresholding nature of decision trees. Since we want to compute explicitly the splitting criterion along each coordinate, we must have a closed expression for the truncated marginals (restriction to any rectangle of $[0,1]^2$), which is not possible with Gaussian distribution. The distribution defined in (\ref{ass:distribution_correlation}) allows us to compute easily both truncated marginals and the joint distribution.

\begin{model}
	\label{ass:last_theorem_corr}
	Let $ \beta \in \mathds{N}$. Assume that 
	$$
	Y = X^{(1)} + X^{(2)} + \alpha X^{(3)} + \varepsilon,
	$$
	where $(X^{(1)}, X^{(2)}) \sim \unifbeta$, $X^{(3)} \sim \mathcal{U}([0,1])$ is independent of $(X^{(1)}, X^{(2)})$, and $\varepsilon$ is an independent noise distributed as $\mathcal{N}(0, \sigma^2)$.
\end{model}

In the framework of Model~\ref{ass:last_theorem_corr}, Proposition~\ref{prop:cor_vs_uncor_input} below states that the CART-split criterion has an explicit expression and highlights that splits along positively correlated variables ($X^{(1)}$ and $X^{(2)}$) are more likely to occur compared to splits along independent ones ($X^{(3)}$). Even if the model is very simple, it is the first theoretical proof that CART splitting procedure tends to favor positively correlated variables. Note however that considering negatively correlated variables will result in an opposite effect, i.e. a tendency to favour independent variables compared to negatively correlated ones.

\begin{proposition}
	\label{prop:cor_vs_uncor_input}
	Assume that Model~\ref{ass:last_theorem_corr} holds. Then, the following statements hold: 
	\begin{itemize}
		\item[$(i)$] The splitting criterion has an explicit expression. For $\beta = 0, \hdots, 5$, each split is performed at the center of the support of the chosen variable. 
		
		\item[$(ii)$] For $\beta = 0, \hdots , 5$, the split in the root node $[0,1]^3$ is performed along $X^{(1)}$ or $X^{(2)}$ if, and only if, $|\alpha| \leq 2$. In that case, the variance reduction is equal to $1/4$.
		
	\end{itemize}
\end{proposition}

Statement $(ii)$ in Proposition~\ref{prop:cor_vs_uncor_input} proves that a positive correlation between two variables increases the probability to split along one of these two variables. Indeed, in our particular model, the marginal effect of $X^{(3)}$ on the output must be at least twice the effect of $X^{(1)}$ or $X^{(2)}$ in order to split along $X^{(3)}$. It is likely that the way the correlation impact the variable selection in trees depends both on the sign of the correlation and on the signs of coefficients in the linear model. Nevertheless, Proposition~\ref{prop:cor_vs_uncor_input} proves that input dependence has an  influence on variable selection in CART procedures. Theorem~\ref{thm:mdi_with_correlation} provides the limiting MDI values given by theoretical and empirical CART trees, in Model~\ref{ass:last_theorem_corr}.

\begin{theorem}
	\label{thm:mdi_with_correlation}
	Let $ \beta \in \{0, \hdots, 5\}$. Assume that Model~\ref{ass:last_theorem_corr} holds. Let $\tree^{\star}$ be any theoretical CART tree and $\tree_n$ be the empirical CART tree. Then, for all $\gamma >0$, there exists $K$ such that, for all $k>K$,
	\begin{align*}
			\Big|   \mdi_{\tree_k^{\star}}(X^{(1)}) + \mdi_{\tree_k^{\star}}(X^{(2)}) - \mathds{V}[X^{(1)} + X^{(2)}] \Big| \leq \gamma.
	\end{align*}
			and 
	\begin{align*}
			\Big|   \mdi_{\tree_k^{\star}}(X^{(3)})  - \alpha^2 \mathds{V}[X^{(3)}] \Big| \leq \gamma.
	\end{align*}
	Additionally, for all $\gamma>0, \rho \in (0,1]$, there exists $K$ such that, for all $k>K$, for all $n$ large enough, with probability at least $1 - \rho$,
			\begin{align*}
			\Big|   \empmdi_{\tree_{n,k}}(X^{(1)}) + \empmdi_{\tree_{n,k}}(X^{(2)}) - \mathds{V}[X^{(1)} + X^{(2)}] \Big| \leq \gamma.
			\end{align*}
			and 
			\begin{align*}
			\Big|   \empmdi_{\tree_{n,k}}(X^{(3)})  - \alpha^2 \mathds{V}[X^{(3)}] \Big| \leq \gamma.
			\end{align*}
\end{theorem}

Theorem~\ref{thm:mdi_with_correlation} gives the limiting values for the MDI. Unfortunately, since $X^{(1)}$ and $X^{(2)}$ are correlated, it is only possible to have information on the sum of the two MDI of $X^{(1)}$ and $X^{(2)}$ rather than on each one of them. According to Lemma~\ref{lemma:correlationX1X2}, a simple calculation shows that the sum of importances of $X^{(1)}$ and $X^{(2)}$ satisfies, 
\begin{align*}
\mathds{V}[X^{(1)} + X^{(2)}] 
%& = \mathds{V}[X^{(1)}] + \mathds{V}[X^{(2)}] + 2 \sqrt{ \mathds{V}[X^{(1)}] \mathds{V}[X^{(2)}]} \textrm{Corr}(X^{(1)}, X^{(2)}) \\
& = 2 \mathds{V}[X^{(1)}]  + 2  \mathds{V}[X^{(1)}] \Big( 1 - \frac{1}{2^{2\beta}} \Big)\\
& > \mathds{V}[X^{(1)}] + \mathds{V}[X^{(2)}].
\end{align*}
Therefore, the group of variables $X^{(1)}$ and $X^{(2)}$ has a larger importance because of their positive correlation. In this case, the amount of information provided by the two variables is larger than that provided by each one of them. This is a very important difference compared to the independent additive case, in which the sum of contributions of each variable is equal to the contributions of all variables. Here, even without interaction effects, this property does not hold, which leads to larger MDI for variables that are positively correlated. As mentioned before, the opposite effect would hold if the variables had a negative correlation or if they had coefficients of opposite signs in the linear model. The impact of correlation has been thoroughly investigated empirically for MDI \cite{nicodemus2009predictor} and MDA \cite{StBoKnAuZe08} but we have theoretically shown above under Model~\ref{ass:last_theorem_corr} that correlation changes the MDI.  

Even if the limiting value $
\mdi_{\tree_K^{\star}}(X^{(1)}) + \mdi_{\tree_K^{\star}}(X^{(2)}) 
$
is known, we have no information on how this quantity is split between variable $X^{(1)}$ and $X^{(2)}$ inside the tree $\tree_K^{\star}$. Indeed, we can find two theoretical trees which produce very different MDI as stated in Lemma~\ref{lem:diff_imp_corr}.
\begin{lemme} \label{lem:diff_imp_corr}
	Let $ \beta \in \{0, \hdots, 5\}$. Assume that Model~\ref{ass:last_theorem_corr} holds. Then, there exists two theoretical trees $\mathcal{T}_{1}$ and $\mathcal{T}_{2}$ such that
	\begin{align*}
	\lim\limits_{k \to \infty} \Big( \mdi_{\mathcal{T}_{2,k}}(X^{(1)}) - \mdi_{\mathcal{T}_{1,k}}(X^{(1)}) \big) = \frac{1}{3} - \frac{1}{3} \Big( \frac{1}{4}\Big)^{\beta}.	
	\end{align*} 
\end{lemme}

Even in the case of a simple linear model with correlated input, it is not wise to use the importance of a single tree to measure the impact of this variable on the output, since this measure can vary drastically between two different trees. Instead, one should rather use an average of this measure across many shallow trees, hoping that randomizing the eligible variables for splitting ensure enough tree diversity in the forest, which in turn, leads to more balanced variable importances. 

The proof of Lemma~\ref{lem:diff_imp_corr} relies on exhibiting a tree whose early splits are made along $X^{(1)}$ only, until $\bX$ is uniformly distributed in each cell. For this tree, the importance of $X^{(1)}$ is larger than the importance of $X^{(2)}$ Since $X^{(1)}$ and $X^{(2)}$ have symmetric roles, one can also build a tree whose early splits are made along $X^{(2)}$ only. The difference of importance for $X^{(1)}$ between these two trees can be computed exactly, thanks to statement $(i)$ and $(ii)$ in Proposition~\ref{prop:cor_vs_uncor_input}. 

\paragraph{Acknowledgements.} We greatly thank the associate editor and the two referees for their careful reading and numerous insightful comments that strongly helped to improve the quality and the readability of the manuscript.

\section{Proofs}
\label{sec:proofs}

\begin{proof}[Proof of Proposition~\ref{thm:sumImportance}]
	Let $\tree^{\star}$ be a theoretical CART tree and define, for all $k \in \mathds{N}$, the variance $u_k$ of the tree $\tree^{\star}_k$ (the truncation of $\tree$ at level $k$) as
	%Throughout the proof, we modify tree construction as follows: for any tree of level $k$, if a cell is not split $k$ times, we copy this cell down the tree (i.e we artificially create a split on the side of the cell), so that terminal nodes in the tree of level $k$ have all been cut exactly $k$ times. Following this construction, we define for any $k$, the variance of the $k$th level of the tree as
	\begin{align*}
	u_k = \sum_{{\ell}=1}^{n_{cell,k}} \mathds{P}[\bX \in A_{\ell, k}] \mathds{V}[Y | \bX \in A_{\ell, k} ],
	\end{align*}
	where $n_{cell,k}$ stands for the number of terminal nodes in $\tree_k^{\star}$ and $\{ A_{\ell,k}, \ell = 1, \hdots, n_{cell,k}\}$ is the set of terminal nodes in $\tree_k^{\star}$. Let $f_k(\bx) = \mathds{V}[Y|\bX' \in A_k(\bx)]$. 
	By definition of the impurity measure, note that, for all $k$, the quantity $u_k - u_{k+1}$ corresponds to the reduction of impurity between level $k$ and level $k+1$. Therefore, the impurity measure rewrites, for any $K \in \mathds{N}$ as
	\begin{align}
	\sum_{j=1}^d \mdi_{\tree_K^{\star}}(X^{(j)}) & = \sum_{k=0}^{K-1} (u_k - u_{k+1}) \nonumber\\
	& = u_0 - u_K, 
	\end{align}
	where $u_0 = \mathds{V}[Y]$ by definition. Additionally, 
	%and $f(\bx) = \mathds{V}[Y|\bX \in A_{\infty}(\bx)]$.
	\begin{align*}
	\mathds{E}[f_k(\bX)] & =  \mathds{E} \Big[ \sum_{{\ell}=1}^{n_{cell,k}} \mathds{V}[Y|\bX' \in A_{\ell, k}] \mathds{1}_{\bX \in A_{\ell, k}} \Big]\\
	& = \sum_{{\ell}=1}^{n_{cell,k}} \mathds{P}[\bX \in A_{\ell, k}]  \mathds{V}[Y|\bX' \in A_{\ell, k}]\\
	& = u_k,
	\end{align*}
	yielding the desired conclusion
	\begin{align}
	\sum_{j=1}^d \mdi_{\tree_K^{\star}}(X^{(j)}) 
	& = \mathds{V}[Y] - \mathds{E}[f_K(\bX)]. \label{eq:proof_summdi_pop}
	\end{align}
\end{proof}

\begin{proof}[Proof of Corollary~\ref{cor:sumImportance}]
Using notations of the proof of Proposition~\ref{thm:sumImportance}, for all $\bx \in [0,1]^d$, $	f_k(\bx)  \leq 4 \|m\|_{\infty}^2 + \sigma^2$ and, by assumption, almost surely, $\lim_{k \to \infty} f_k(\bX) = \sigma^2$. According to the dominated convergence theorem,
\begin{align*}
\lim\limits_{k \to \infty} u_k = \lim\limits_{k \to \infty} \mathds{E}[f_k(\bX)]  = \sigma^2.
\end{align*}
Therefore, 
\begin{align}
\lim\limits_{K \to \infty} \sum_{j=1}^d \mdi_{\tree_K^{\star}}(X^{(j)}) 
& = \mathds{V}[Y] - \sigma^2.
\end{align} 
\end{proof}

\begin{proof}[Proof of Proposition~\ref{thm:sumImportance_emp}]
 Equation~\eqref{eq:proof_summdi_pop} in the proof of Proposition~\ref{thm:sumImportance}, written with empirical quantities, leads directly to the first statement. 	
\end{proof}

\begin{proof}[Proof of Theorem~\ref{thm:sumImportance_emp2}]
	
	\item 
 \paragraph{Uniform convergence of theoretical trees.}
 Note that a tree is nothing but a sequence of splits. Therefore, any tree $\tree$ can be seen as a sequence $(v_n)$ such that
 \begin{itemize}
 	\item $(v_1, \hdots, v_d)$ represents the first split of the tree defined, for a split occurring along variable $j$ at position $s$, as 
 	\begin{align*}
 	\left\lbrace 
 	\begin{array}{ll}
 	v_{\ell} = 0 & \textrm{for} ~\ell \neq j \\
	v_j = \frac{s}{d^4}
 	\end{array}
 	\right.
 	\end{align*}
 	
 	\item Each block $(v_{md+1}, \hdots , v_{(m+1)d})$ represents the $m$th split of the tree occurring along variable $j$ at position $s$ defined as
 	\begin{align*}
 	\left\lbrace 
 	\begin{array}{ll}
 	v_{\ell} = 0 & \textrm{for} ~\ell \neq j \\
 	v_j = \frac{s}{(d(m+1))^4}
 	\end{array}
 	\right.
 	\end{align*}
 	
 	\item Splits are ordered arbitrarily. 
 	
 	\item In order to obtain a complete tree, we associate with all nodes whose construction is stopped, the split $(0, \hdots , 0)$. 
 \end{itemize} 
Since the split position belongs to $[0,1]$, we have $\sum_{{\ell}=1}^{\infty}  v_{\ell}^2 < \infty$. Besides, 
\begin{align*}
\sum_{{\ell}=1}^{\infty} \ell^2 v_{\ell}^2 \leq \sum_{{\ell}=1}^{\infty} \frac{1}{\ell^2} = \frac{\pi^2}{6}.
\end{align*}
For any tree $\tree$ represented by the sequence $(v_n)$, we have 
$$(v_n) \in B(0, \pi^2/6) = \big\lbrace v, \sum_{{\ell}=1}^{\infty} \ell^2 v_{\ell}^2 \leq \pi^2/6 \big\rbrace,
$$
 with $B(0, \pi^2/6)$ being a compact of $\ell^2(\mathds{N})$. Now, let $\gamma >0$ and
 \begin{align*}
\mathcal{A}_{k, \gamma} = \big\lbrace \tree^{\star}, \mathds{E}[\Delta (m, A_{\tree_k^{\star}}(\bX))] \geq \gamma \big\rbrace
 \end{align*}  
Set $\mathcal{A}_{\infty,\gamma} = \cap_{k=0}^{\infty} \mathcal{A}_{k, \gamma}$. Since the population CART-split criterion $L^{\star}$, defined in \eqref{eq:theoretical_cart_crit} is non negative, we have
\begin{align*}
\mathds{E}[\Delta (m, A_{\tree_{k+1}}(\bX))] \leq \mathds{E}[\Delta (m, A_{\tree_k}(\bX))].
\end{align*}
Therefore, for all $k$, $\mathcal{A}_{k+1, \gamma} \subset \mathcal{A}_{k, \gamma}$. Because of the uniform continuity of the splitting criterion (continuous on a compact), for all $k \in \mathds{N}$, the sets $\mathcal{A}_{k, \gamma}$ are closed.

Assume that for all $k$, $\mathcal{A}_{k, \gamma} \neq \emptyset$. We know that $B(0, \pi^2/6)$ is compact and that, for all $k$, $\mathcal{A}_{k, \gamma}$ is a closed subset of $B(0, \pi^2/6)$. Then, according to Cantor's intersection theorem $\mathcal{A}_{\infty, \gamma} \neq \emptyset$. Consequently, there exists $\tree^{\star} \in \mathcal{A}_{\infty, \gamma}$. By definition of $\mathcal{A}_{\infty,\gamma}$, for all $k$, 
\begin{align}
\mathds{E}[\Delta (m, A_{\tree_k^{\star}}(\bX))] \geq \gamma. \label{eq:proof_absurd}
\end{align}
Again, by definition of $\mathcal{A}_{\infty,\gamma}$, $\tree^{\star}$ is a theoretical tree, which is consistent by assumption. Consequently, there exists $K$ such that, for all $k > K$, 
\begin{align*}
\mathds{E}[\Delta (m, A_{\tree_k^{\star}}(\bX))] \leq \gamma,
\end{align*}
which is absurd, according to equation \eqref{eq:proof_absurd}.  Therefore, there exists $K>0$ such that $\mathcal{A}_{K, \gamma} = \emptyset$. All in all, we proved that, for all $\gamma>0$, there exists $K>0$ such that, for all theoretical tree $\tree^{\star}$, for all $k >K$, 
\begin{align}
\mathds{E}[\Delta (m, A_{\tree_{k}}(\bX))] < \gamma. \label{ineq:theoretical_tree_uniform}
\end{align}

\paragraph{Contribution of cells with small measure.}

Fix $\gamma$ and take $K$ such that inequality (\ref{ineq:theoretical_tree_uniform}) holds. For any theoretical tree $\tree_K^{\star}$ of level $K$, the MDI associated with this tree is defined as
\begin{align*}
\mdi_{\tree^{\star}_K} = \sum_{A \in \tree^{\star}_K} \mu (A) L^{\star}_A,
\end{align*}
where the sum ranges for any node $A$ in the tree, $\mu(A)$ is the measure of the cell $A$ defined as 
\begin{align*}
\mu(A) = \int_{x \in A} f(\bx) \textrm{d}\bx,
\end{align*}
where $f$ is the density of $\bX$ and $L^{\star}_A$ the value of the population CART-split criterion at the best split of the cell $A$. Similarly, we let
\begin{align*}
\empmdi_{\tree_{n,K}} = \sum_{A \in \tree_{n,K}} \mu_n (A) L_{n,A},
\end{align*}
where $\tree_{n,K}$ is the empirical CART tree truncated at level $K$, $A$ a node of the tree, $\mu_n (A)$ the empirical measure of $A$ and $L_{n,A}$ the empirical splitting criterion in cell $A$. 
%\begin{align*}
%\inf_{\tree^{\star}_K} |\empmdi_{\tree_{n,K}} - \mdi_{\tree^{\star}_K}|& \leq |\empmdi_{\tree_{n,K}} - \empmdi_{\tree_{n,K}, \gamma} |  + \inf_{\tree^{\star}_K} |\empmdi_{\tree_{n,K}, \gamma} - \mdi_{\tree^{\star}_K, \gamma}|  \\
%& \quad + \inf_{\tree^{\star}_K}|\mdi_{\tree^{\star}_K, \gamma} - \mdi_{\tree^{\star}_K}| 
%\end{align*}
Let $\mdi_{\tree^{\star}_K, \gamma }$ and $\empmdi_{\tree_{n,K}, \gamma}$ be computed by removing the variance (empirical or population version) associated to cells $A$ for which $\mu(A) < \gamma$. Define also $\mdi_{\tree_{n,K}, \gamma}$ as the population version of the MDI but computed with the empirical partition $\tree_{n,K}$ where the cells $A$ for which $\mu(A) < \gamma$ have been removed. Now, we want to control
\begin{align}
\inf_{\tree^{\star}} |\empmdi_{\tree_{n,K}} - \mdi_{\tree^{\star}_K}|
& \leq |\empmdi_{\tree_{n,K}} - \empmdi_{\tree_{n,K}, \gamma} | + | \empmdi_{\tree_{n,K}, \gamma} - \mdi_{\tree_{n,K}, \gamma}  |  
\nonumber \\
&  \quad  + \inf_{\tree^{\star}}
| \mdi_{\tree_{n,K}, \gamma} - \mdi_{\tree^{\star}_K, \gamma}| + \inf_{\tree^{\star}} | \mdi_{\tree^{\star}_K, \gamma} - \mdi_{\tree^{\star}_K}|. \label{eq:main_proof_decomp}
\end{align}  
Since $\bX$ admits a density on $[0,1]^d$, the population CART-split criterion varies continuously as a function of the cell on which it is computed. Thus, the third term in (\ref{eq:main_proof_decomp}) is controlled as soon as the empirical partition is close to the theoretical partition. Lemma $3$ in \cite{ScBiVe15} can be easily extended for any continuous regression function $m$ and any density which is bounded from above and below by some positive constants. Therefore, with probability $1-\rho$, for all $n$ large enough, 
\begin{align}
\inf_{\tree^{\star}}
|\mdi_{\tree_{n,K}, \gamma} - \mdi_{\tree^{\star}_K, \gamma}| \leq \gamma. \label{eq:main_proof_decomp1}
\end{align}
Last term in (\ref{eq:main_proof_decomp}) can be treated directly by noting that the theoretical criterion is bounded above by $4 \|m\|_{\infty}^2$. Therefore, we have
\begin{align}
|\mdi_{\tree^{\star}_K, \gamma} - \mdi_{\tree^{\star}_K}| \leq  2^{K+2} \|m\|_{\infty}^2 \gamma. \label{eq:main_proof_decomp2}
\end{align}
To control the second term, we make use of the results of \cite{wager2015adaptive}, in particular Corollary $11$ which can be adapted to a Gaussian noise by noting that with probability $1-\rho$, $\max_{1\leq i \leq n} |\varepsilon_i| \leq C_{\rho}\log n$. Note that our setting is simpler than that of \cite{wager2015adaptive} where the dimension $d$ grows to infinity with $n$. Therefore, with probability at least $1-\rho$, for all $n$ large enough,
\begin{align}
|\empmdi_{\tree_{n,K}, \gamma} - \mdi_{\tree_{n,K}, \gamma} | \leq \gamma.\label{eq:main_proof_decomp3}
\end{align}
Regarding the first term in (\ref{eq:main_proof_decomp}), for any cell $A$ such that $\mu(A) < \gamma$, 
\begin{align*}
\mu_n(A) \mathds{V}_n[Y | X \in A] \leq 4\mu_n(A) \|m\|_{\infty}^2 +  \mu_n(A) \mathds{V}_n[\varepsilon | X \in A].
\end{align*}
With probability $1-\rho$, for all $n$ large enough, $\max\limits_{1\leq i \leq n} |\varepsilon_i| \leq C_{\rho}\log n$ and, if $n \mu_n(A) \geq \sqrt{n}$, 
\begin{align*}
\frac{1}{n \mu_n(A)} \sum_{i, X_i \in A} \varepsilon_i^2 \leq 2\sigma^2.
\end{align*}
Therefore, with probability $1-\rho$, for all $n$ large enough, 
\begin{align}
4\mu_n(A) \|m\|_{\infty}^2 + \mu_n(A) \mathds{V}_n[\varepsilon | X \in A] 
& \leq 8 \gamma \|m\|_{\infty}^2 + \frac{N_n(A) }{n} \frac{1}{N_n(A)} \sum_{i, X_i \in A} \varepsilon_i^2 \nonumber \\
& \leq 8 \gamma \|m\|_{\infty}^2 + 4 \gamma \sigma^2 \mathds{1}_{N_n(A) \geq \sqrt{n}} + \frac{C_{\rho} \log n }{\sqrt{n}} \mathds{1}_{N_n(A) < \sqrt{n}} \nonumber \\
& \leq 8 \gamma \|m\|_{\infty}^2 + 4 \gamma \sigma^2. \label{eq:main_proof_decomp4}
\end{align}
Injecting inequalities (\ref{eq:main_proof_decomp1})-(\ref{eq:main_proof_decomp4}) into (\ref{eq:main_proof_decomp}), with probability $1- \rho$, for all $n$ large enough, 
\begin{align}
\inf_{\tree^{\star}} |\empmdi_{\tree_{n,K}} - \mdi_{\tree^{\star}_K}|
& \leq 8 \gamma \|m\|_{\infty}^2 + 4 \gamma \sigma^2 + 2\gamma    + 2^{K+2} \|m\|_{\infty}^2 \gamma. \label{eq:final_decomp_proof}
\end{align}  
Since, 
\begin{align*}
|\empmdi_{\tree_{n,K}} - \mathds{V}[m(\bX)]| \leq \inf_{\tree^{\star}} |\empmdi_{\tree_{n,K}} - \mdi_{\tree^{\star}_K}| +  | \mdi_{\tree^{\star}_K} - \mathds{V}[m(\bX)]|,
\end{align*}
we conclude, from (\ref{ineq:theoretical_tree_uniform}), (\ref{eq:final_decomp_proof}) and Corollary~\ref{cor:sumImportance}, with probability $1 - \rho$, for all $n$ large enough,
\begin{align*}
|\empmdi_{\tree_{n,K}} - \mathds{V}[m(\bX)]| \leq 3\gamma  + 8 \gamma \|m\|_{\infty}^2 + 4 \gamma \sigma^2 + 2^{K+2} \|m\|_{\infty}^2 \gamma.
\end{align*}
Since $\rho$ and $\gamma$ are arbitrary, this concludes the proof. 

\end{proof}

\begin{proof}[Proof of Lemma~\ref{Lem:decompt_th_3}]

Let $A = \prod_{j=1}^d [a_j, b_j] \subset [0,1]^d$ be a generic cell. Recall that the regression model satisfies 
\begin{align*}
Y = m_{\mathcal{J}}(\bX^{(\mathcal{J})}) + m_{\mathcal{J}^c}(\bX^{(\mathcal{J}^c)})+ \varepsilon.
\end{align*}
For all $j \in \mathcal{J}$, and all splitting position $s$, define the two resulting cells 
\begin{align*}
A_L = A \cap \{\bx, x^{(j)} < s\} \quad \textrm{and} \quad  A_R = A \cap \{\bx, x^{(j)} \geq s\},
\end{align*} 
where, for notation brevity, the dependence on $j$ and $s$ is omitted. 	The variance reduction associated to any split $(j,s)$ is defined as 
\begin{align*}
L^{\star}_{A}(j,s) & =  \mathds{V}[m(\bX) | \bX \in A ] - \mathds{P}[\bX \in A_L | \bX \in A] \mathds{V}[m(\bX) | \bX \in A_L ] \\
& \quad - \mathds{P}[\bX \in A_R | \bX \in A] \mathds{V}[m(\bX) | \bX \in A_R ]\\
&  =  \mathds{V}[m_{\mathcal{J}}(\bX^{(\mathcal{J})}) | \bX \in A ] - \mathds{P}[\bX \in A_L | \bX \in A] \mathds{V}[m_{\mathcal{J}}(\bX^{(\mathcal{J})}) | \bX \in A_L ] \\
& \quad - \mathds{P}[\bX \in A_R | \bX \in A ] \mathds{V}[m_{\mathcal{J}}(\bX^{(\mathcal{J})}) | \bX \in A_R ]\\
& \quad +  \mathds{V}[m_{\mathcal{J}^c}(\bX^{(\mathcal{J}^c)}) | \bX \in A ] - \mathds{P}[\bX \in A_L| \bX \in A] \mathds{V}[m_{\mathcal{J}^c}(\bX^{(\mathcal{J}^c)}) | \bX \in A_L ] \\
& \quad - \mathds{P}[\bX \in A_R| \bX \in A] \mathds{V}[m_{\mathcal{J}^c}(\bX^{(\mathcal{J}^c)})| \bX \in A_R ].
\end{align*}
Since the split is performed along the $j$th coordinate, 
\begin{align}
\mathds{V}[m_{\mathcal{J}^c}(\bX^{(\mathcal{J}^c)}) | \bX \in A_L ] = \mathds{V}[m_{\mathcal{J}^c}(\bX^{(\mathcal{J}^c)}) | \bX \in A_R ] = \mathds{V}[m_{\mathcal{J}^c}(\bX^{(\mathcal{J}^c)}) | \bX \in A ]. \label{eq:crit_indep}
\end{align}
Consequently, the variance reduction satisfies 
\begin{align}
L^{\star}_{A}(j,s) 
& =  \mathds{V}[m_{\mathcal{J}}(\bX^{(\mathcal{J})}) | \bX \in A ] - \mathds{P}[\bX \in A_L  | \bX \in A] \mathds{V}[m_{\mathcal{J}}(\bX^{(\mathcal{J})}) | \bX \in A_L ] \nonumber \\
& \quad - \mathds{P}[\bX \in A_R | \bX \in A] \mathds{V}[m_{\mathcal{J}}(\bX^{(\mathcal{J})}) | \bX \in A_R ]. \label{eq:crit_rewritting}
\end{align}
Since $\bX^{(\mathcal{J})}$ and $\bX^{(\mathcal{J})}$ are independent, letting $A^{(\mathcal{J})} = \prod_{j \in \mathcal{J}} [a_j, b_j]$ be the projection of the cell $A$ along coordinates in $\mathcal{J}$, we have
\begin{align}
L^{\star}_{A}(j,s) 
& = L^{\star}_{A^{(\mathcal{J})}}(j,s)\nonumber \\
& =  \mathds{V}[m_{\mathcal{J}}(\bX^{(\mathcal{J})}) | \bX \in A^{(\mathcal{J})} ] - \mathds{P}[\bX \in A^{(\mathcal{J})}_L  | \bX \in A^{(\mathcal{J})}] \mathds{V}[m_{\mathcal{J}}(\bX^{(\mathcal{J})}) | \bX \in A^{(\mathcal{J})}_L ] \nonumber \\
& \quad - \mathds{P}[\bX \in A^{(\mathcal{J})}_R | \bX \in A^{(\mathcal{J})}] \mathds{V}[m_{\mathcal{J}}(\bX^{(\mathcal{J})}) | \bX \in A^{(\mathcal{J})}_R ]. \label{eq:crit_rewritting2}
\end{align}
Therefore, for all $j \in \mathcal{J}$, the criterion $L^{\star}_{A}(j,s) $ does not depend on $m_{\mathcal{J}^c}$ and is equal to its restriction over the cell $A^{(\mathcal{J})}$. Besides, according to equation~(\ref{eq:crit_rewritting}), any split $j \in \mathcal{J}$ does not change the variance of $m_{\mathcal{J}^c}$.

\end{proof}

\begin{proof}[Proof of Proposition~\ref{prop:th_decomp_th3}]
	
Let $\bx \in [0,1]^d$ such that 
\begin{itemize}
	\item[$(i)$] any theoretical tree $\tree^{\star}_{k,\mathcal{J}}$ built on the distribution $(\bX^{(\mathcal{J})}, m_{\mathcal{J}}(\bX^{(\mathcal{J})}))$ is consistent, that is 
	$$
	\lim\limits_{k \to \infty} \Delta(m_{\mathcal{J}} , A^{(\mathcal{J})}_{\tree^{\star}_{k,\mathcal{J}}}(\bx^{(\mathcal{J})}) ) = 0.
	$$
	
	\item[$(ii)$] any theoretical tree $\tree^{\star}_{k,\mathcal{J}^c}$ built on the distribution $(\bX^{(\mathcal{J}^c)}, m_{\mathcal{J}^c}(\bX^{(\mathcal{J}^c)}))$ is consistent, that is 
	$$
	\lim\limits_{k \to \infty} \Delta(m_{\mathcal{J}^c} , A^{(\mathcal{J}^c)}_{\tree^{\star}_{k,\mathcal{J}^c}}(\bx^{(\mathcal{J}^c)}) ) = 0.
	$$
\end{itemize}
Consider a theoretical tree $\tree^{\star}_{k}$ built on the distribution $(\bX, Y)$. For all $k$, let $A_k(\bx)$ be the cell in $\tree^{\star}_{k}$ containing $\bx$. For all $k$, we let 
\begin{itemize}
	\item[$(i)$] $A_{\phi(k)}(\bx)$ be the subsequence of $A_k(\bx)$ composed of cells split by a coordinate in $\mathcal{J}$, and 
	$A^{(\mathcal{J})}_{\phi(k)}(\bx)$ the projection of $A_{\phi(k)}(\bx)$ over $\mathcal{J}$,
	
	\item[$(ii)$] $A_{\psi(k)}(\bx)$ be the subsequence of $A_k(\bx)$ composed of cells split by a coordinate in $\mathcal{J}^c$, and $A^{(\mathcal{J}^c)}_{\psi(k)}(\bx)$ the projection of $A_{\psi(k)}(\bx)$ over $\mathcal{J}^c$.
\end{itemize}

According to Lemma~\ref{Lem:decompt_th_3}, the restrictions of the function $m_{\mathcal{J}}$ over $A \in \{A^{(\mathcal{J})}_{\phi(k)+1}(\bx), $ $ \hdots , A^{(\mathcal{J})}_{\phi(k+1)}(\bx) \}$ are equal. Besides, for any $j \in \mathcal{J}$ and $s$, the splitting criteria $L^{\star}_{A}(j,s)$ for $A \in \{A^{(\mathcal{J})}_{\phi(k)+1}(\bx), \hdots , A^{(\mathcal{J})}_{\phi(k+1)}(\bx) \}$ are the same. 
%is constant in all cells $A^{\mathcal{J}}_{\phi(k)+1}(\bx), \hdots , A^{\mathcal{J}}_{\phi(k+1)}(\bx)$. Besides, for all $j \in \mathcal{J}$, the splitting criterion $\mathcal{L}_{A}(j,s)$ does not depend on $A \in \{A^{\mathcal{J}}_{\phi(k)+1}(\bx), \hdots , A^{\mathcal{J}}_{\phi(k+1)}(\bx) \}$. 
Therefore, the sequence $A^{(\mathcal{J})}_{\phi(k)}(\bx)$ is a sequence of cells resulting from a theoretical tree built on the distribution $(\bX^{(\mathcal{J})}, m_{\mathcal{J}}(\bX^{(\mathcal{J})}))$. Similarly, the sequence $A^{(\mathcal{J}^c)}_{\psi(k)}(\bx)$ is a sequence of cells resulting from a theoretical tree built on the distribution $(\bX^{(\mathcal{J}^c)}, m_{\mathcal{J}^c}(\bX^{(\mathcal{J}^c)}))$.

Let $\tree^{\star}_{\mathcal{J}}$ be any theoretical tree whose sequence of cells satisfies $	A_{\tree^{\star}_{k,\mathcal{J}}} (\bx^{(\mathcal{J})}) = A^{(\mathcal{J})}_{\phi(k)}(\bx)$.
Similarly, let 
$\tree^{\star}_{\mathcal{J}^c}$ be any theoretical tree whose sequence of cells satisfies $
A_{\tree^{\star}_{k,\mathcal{J}^c}} (\bx^{(\mathcal{J}^c)}) = A^{(\mathcal{J}^c)}_{\psi(k)}(\bx)$.
Then, 
\begin{align}
\Delta(m , A_{\tree^{\star}_{k}}(\bx) ) & =  \mathds{V}[m(\bX) | \bX \in A_{\tree^{\star}_{k}}(\bx)] \nonumber \\
%& = \mathds{V}[m_{\mathcal{J}}(\bX) | \bX \in A_{\tree^{\star}_{k}}(\bx)] +  \mathds{V}[m_{\mathcal{J}^c}(\bX) | \bX \in A_{\tree^{\star}_{k}}(\bx)] \nonumber \\
& = \mathds{V}[m_{\mathcal{J}}(\bX^{(\mathcal{J})}) | \bX \in A^{(\mathcal{J})}_{\tree^{\star}_{k}}(\bx)] +  \mathds{V}[m_{\mathcal{J}^c}(\bX^{(\mathcal{J}^c)})) | \bX \in A^{(\mathcal{J}^c)}_{\tree^{\star}_{k}}(\bx)] \nonumber \\
& = \Delta(m_{\mathcal{J}} , A^{(\mathcal{J})}_{\tree^{\star}_{k,\mathcal{J}}}(\bx^{(\mathcal{J})}) ) 
+ \Delta(m_{\mathcal{J}^c} , A^{(\mathcal{J}^c)}_{\tree^{\star}_{k,\mathcal{J}^c}}(\bx^{(\mathcal{J}^c)}) ). \label{eq:proof_decomp_var_finish}
\end{align}
By assumption, and assuming that $\lim\limits_{\infty} \phi, \psi = \infty$, we have  
\begin{align*}
\lim\limits_{k \to \infty} \Delta(m , A_{\tree^{\star}_{k}}(\bx) ) = 0.
\end{align*}
Now, consider the case where $\lim\limits_{\infty} \phi, \psi \neq \infty$. Since $\phi(\mathds{N}) \cup \psi(\mathds{N}) = \mathds{N}$, one can assume, without loss of generality, that $\lim\limits_{\infty} \phi = \infty$ and $\lim\limits_{\infty} \psi = p < \infty$. Thus, there exists $K$ such that for all $k>K$, splits in $A_{\tree^{\star}_{k}}(\bx)$ are performed along coordinates in $\mathcal{J}$ only. Thus, by assumption, the corresponding theoretical tree $\tree^{\star}_{k,\mathcal{J}}$ is consistent which implies that the splitting criterion $\mathcal{L}_{A_k(\bx)}(j,s)$ tends to zero, for all $j \in \mathcal{J}$. This implies that the criterion $\mathcal{L}_{A_k(\bx)}(j,s)$, for all $j \in \mathcal{J}^c$ is equal to zero (otherwise, it would be larger than the criterion $\mathcal{L}_{A_k(\bx)}(j,s)$ for $j \in \mathcal{J}$ and consequently, splits would be performed along $j \in \mathcal{J}^c$). This implies that the theoretical tree $\tree^{\star}_{K,\mathcal{J}^c}$ is fully grown, and by assumption, we have
\begin{align*}
\Delta(m_{\mathcal{J}^c} , A^{(\mathcal{J}^c)}_{\tree^{\star}_{K,\mathcal{J}^c}}(\bx^{(\mathcal{J}^c)}) ) = 0,
\end{align*} 
which concludes the proof, by equality \eqref{eq:proof_decomp_var_finish}.

\end{proof}

\begin{proof}[Proof of Theorem \ref{thm:mdi_interaction}\\]

	\item 
	\paragraph{First statement.}
	For any theoretical CART tree $\tree^{\star}$, we let, for all $k \in \mathds{N}$, 
	%the variance $u_k$ of the tree $\tree^{\star}_k$ (the truncation of $\tree$ at level $k$) as
	%Throughout the proof, we modify tree construction as follows: for any tree of level $k$, if a cell is not split $k$ times, we copy this cell down the tree (i.e we artificially create a split on the side of the cell), so that terminal nodes in the tree of level $k$ have all been cut exactly $k$ times. Following this construction, we define for any $k$, the variance of the $k$th level of the tree as
%	\begin{align*}
%	u_k = \sum_{{\ell}=1}^{n_{cell,k}} \mathds{P}[\bX \in A_{\ell, k}] \mathds{V}[Y | \bX \in A_{\ell, k} ],
%	\end{align*}
%	where $n_{cell,k}$ stands for the number of terminal nodes in $\tree_k^{\star}$ and $\{ A_{\ell,k}, \ell = 1, \hdots, n_{cell,k}\}$ is the set of terminal nodes in $\tree_k^{\star}$. Let $f_k(\bx) = \mathds{V}[Y|\bX' \in A_k(\bx)]$. 
%	%
	%We define
	\begin{align*}
	u_{k, \mathcal{J}} = \sum_{{\ell}=1}^{n_{cell,k}} \mathds{P}[\bX' \in A_{\ell, k}]  \mathds{V}[m_{\mathcal{J}}(\bX^{(\mathcal{J})})|\bX \in A_{\ell, k}],
	\end{align*}
	and
	\begin{align*}
	f_{k, \mathcal{J}}(\bx) = \mathds{V}[m_{\mathcal{J}}(\bX^{(\mathcal{J})})|\bX \in A_k(\bx)].
	\end{align*}
	We define similarly $u_{k, \mathcal{J}^c}$ and $f_{k, \mathcal{J}^c}(\bx)$. By definition of the regression model, for all $\bx \in [0,1]^d$, 
	\begin{align*}
	f_{k, \mathcal{J}}(\bx) = f_{k}(\bx) - f_{k, \mathcal{J}^c}(\bx) \leq f_{k}(\bx),
	\end{align*}
	where $f_{k}(\bx) = \mathds{V}[m(\bX)|\bX \in A_{k}(\bx)]$. By assumption this quantity tends to zero for almost all $\bX \in [0,1]^d$. Besides, it satisfies $f_{k, \mathcal{J}}(\bx) \leq 4 \|m\|_{\infty}^2$. Therefore, according to the dominated convergence theorem, $\lim\limits_{k \to \infty} \mathds{E}[f_{k, \mathcal{J}}(\bX')] = 0$. Additionally, 
	\begin{align*}
	\mathds{E}[f_{k, \mathcal{J}}(\bX')] & =  \mathds{E} \Big[ \sum_{{\ell}=1}^{n_{cell,k}} \mathds{V}[m_{\mathcal{J}}(\bX^{(\mathcal{J})})|\bX \in A_{\ell, k}] \mathds{1}_{\bX' \in A_{\ell, k}} \Big]\\
	& = \sum_{{\ell}=1}^{n_{cell,k}} \mathds{P}[\bX' \in A_{\ell, k}]  \mathds{V}[m_{\mathcal{J}}(\bX^{(\mathcal{J})})|\bX \in A_{\ell, k}]\\
	& = u_{k, \mathcal{J}}.
	\end{align*}
	Consequently, 
	\begin{align}
	\lim\limits_{k \to \infty} u_{k, \mathcal{J}} = 0 \label{eq:u_j_tend_zero}
	\end{align}
	The same result holds for $u_{k, \mathcal{J}^c}$.
%	\begin{align*}
%	f_{\mathcal{J}}(\bx) = \mathds{V}[m_{\mathcal{J}}(\bX^{(\mathcal{J})})|\bX \in A_{\infty}(\bx)].
%	\end{align*}
%	The terms $u_{k, \mathcal{J}^c}$, $f_{k, \mathcal{J}^c}, f_{\mathcal{J}^c}$ are defined in the exact same way. Note that 
%	According to Lemma~\ref{lemma1}, for all $\bx \in [0,1]^d$, 
%	\begin{align*}
%	\lim\limits_{k \to \infty} \mathds{E}[m_{\mathcal{J}}(\bX^{(\mathcal{J})}) | \bX \in A_k(\bx)] = \mathds{E}[m_{\mathcal{J}}(\bX^{(\mathcal{J})}) | \bX \in A_{\infty}(\bx)],
%	\end{align*}
%	and
%	\begin{align*}
%	\lim\limits_{k \to \infty} \mathds{E}[m_{\mathcal{J}}(\bX^{(\mathcal{J})})^2 | \bX \in A_k(\bx)] = \mathds{E}[m_{\mathcal{J}}(\bX^{(\mathcal{J})})^2 | \bX \in A_{\infty}(\bx)],
%	\end{align*}
%	Thus, for all $\bx \in [0,1]^d$, $\lim\limits_{k \to \infty} f_{k, \mathcal{J}}(\bx) = f(\bx)$.
%	Besides, $\| f_{k, \mathcal{J}}\|_{\infty}  \leq 4 \|m_{\mathcal{J}}\|_{\infty}^2$.
%	Hence, according to the dominated convergence theorem, we have
%	\begin{align*}
%	\lim\limits_{k \to \infty} u_{k,\mathcal{J}} = \lim\limits_{k \to \infty} \mathds{E}[f_{k, \mathcal{J}}(\bX')] =  \mathds{E}[\mathds{V}[m_{\mathcal{J}}(\bX^{(\mathcal{J})})|\bX \in A_{\infty}(\bX')]].
%	\end{align*}
%	Since the first term is upper bounded, we deduce that
%	\begin{align*}
%	\mathds{E}[\mathds{V}[m_{\mathcal{J}}(\bX^{(\mathcal{J})})|\bX \in A_{\infty}(\bX')]] \to 0,
%	\end{align*}
%	that is $\lim\limits_{k \to \infty} u_{k,\mathcal{J}} =0$. 
	%
	To conclude the proof, we just need to show that $u_{k,\mathcal{J}} - u_{k+1,\mathcal{J}}$ corresponds to the variance reduction associated with variables $\bX^{(\mathcal{J})}$ between levels $k$ and $k+1$. Recalling that 
	%Throughout the proof, we modify tree construction as follows: for any tree of level $k$, if a cell is not split $k$ times, we copy this cell down the tree (i.e we artificially create a split on the side of the cell), so that terminal nodes in the tree of level $k$ have all been cut exactly $k$ times. Following this construction, we define for any $k$, the variance of the $k$th level of the tree as
	\begin{align*}
	u_k = \sum_{{\ell}=1}^{n_{cell,k}} \mathds{P}[\bX \in A_{\ell, k}] \mathds{V}[m(\bX) | \bX \in A_{\ell, k} ] + \sigma^2,
	\end{align*}
	the variance reduction between levels $k$ and $k+1$ is equal to 
	\begin{align*}
	u_{k} - u_{k+1} = (u_{k,\mathcal{J}} - u_{k+1,\mathcal{J}}) + (u_{k,\mathcal{J}^c} - u_{k+1,\mathcal{J}^c}),
	\end{align*}
    where
	\begin{align}
	u_{k,\mathcal{J}} - u_{k+1,\mathcal{J}} & = \sum_{{\ell}=1}^{n_{cell,k}} \mathds{V} \big[ m_{\mathcal{J}}(\bX^{(\mathcal{J})}) \big| \bX \in A_{\ell, k} \big] \mathds{P}[\bX' \in A_{\ell, k}] \nonumber \\
	&\qquad - \sum_{{\ell}=1}^{n_{cell,k+1}} \mathds{V} \big[  m_{\mathcal{J}}(\bX^{(\mathcal{J})})  \big| \bX \in A_{\ell, k+1} \big] \mathds{P} [\bX' \in A_{\ell, k}]. \label{criterion_1}
	\end{align}
	Each cell at level $k+1$ has been created $(i)$ either by splitting a cell at level $k$ into two cells, $(ii)$ or by copying a cell at level $k$, if the cell is not split. In that latter case, the contribution of this cell to $u_{k, \mathcal{J}} - u_{k+1, \mathcal{J}}$ is zero. 	Thus, we can rewrite equation (\ref{criterion_1}) as
	\begin{align}
	& u_{k,\mathcal{J}} - u_{k+1,\mathcal{J}} \nonumber \\
	= & \sum_{{\ell}=1}^{n_{\mathcal{J}, k}} 
	\Big( \mathds{P}[\bX' \in A_{\ell, k}] \mathds{V} \big[ m_{\mathcal{J}}(\bX^{(\mathcal{J})}) \big| \bX \in A_{\ell, k} \big]  - \mathds{P}[\bX' \in A_{\ell, k,L}] \mathds{V} \big[ m_{\mathcal{J}}(\bX^{(\mathcal{J})}) \big| \bX \in A_{\ell, k, L} \big]  \nonumber\\
	& \quad - \mathds{P}[\bX' \in A_{\ell, k, R}] \mathds{V} \big[ m_{\mathcal{J}}(\bX^{(\mathcal{J})}) \big| \bX \in A_{\ell, k, R} \big]  \Big)\nonumber\\
	& + \sum_{{\ell}=n_{\mathcal{J}, k}+1}^{n_{split, k}} 
	\Big( \mathds{P}[\bX' \in A_{\ell, k}] \mathds{V} \big[ m_{\mathcal{J}}(\bX^{(\mathcal{J})}) \big| \bX \in A_{\ell, k} \big]  \nonumber \\
	& \quad - \mathds{P}[\bX' \in A_{\ell, k, L}] \mathds{V} \big[ m_{\mathcal{J}}(\bX^{(\mathcal{J})}) \big| \bX \in A_{\ell, k, L} \big] \nonumber \\
	& \quad - \mathds{P}[\bX' \in A_{\ell, k, R}] \mathds{V} \big[ m_{\mathcal{J}}(\bX^{(\mathcal{J})}) \big| \bX \in A_{\ell, k, R} \big]  \Big) \label{eq:decomposition_impurity1}
	\end{align}
	where $A_{\ell, k+1, L}$ and $A_{\ell, k+1, R}$ are the children of $A_{\ell, k}$, the first sum (resp. the second) corresponds to the cell split along a coordinate in $\mathcal{J}$ (resp. in $\mathcal{J}^c$), and $n_{\mathcal{J}, k}$  is the number of split performed along a coordinate in $\mathcal{J}$ between level $k$ and $k+1$.
	
	Note that, according to equation~(\ref{eq:crit_indep}), all terms in the second sum in equation (\ref{eq:decomposition_impurity1}) are null. Finally, according to equation~(\ref{eq:crit_rewritting}), $u_{k,\mathcal{J}} - u_{k+1,\mathcal{J}} $ is exactly the sum of variance reduction for splits along coordinates in $\mathcal{J}$ that occur between level $k$ and $k+1$. Consequently, for all $K$, the population MDI computed with the theoretical tree $\tree_K^{\star}$, 
	\begin{align*}
	\sum_{j \in \mathcal{J}} \mdi_{\tree_K^{\star}}(X^{(j)}) & = \sum_{k=0}^{K} \big( u_{k,\mathcal{J}} - u_{k+1,\mathcal{J}} \big) \\
	& = u_{0,\mathcal{J}} - u_{K,\mathcal{J}}.
	\end{align*} 
	Hence, according to equation \eqref{eq:u_j_tend_zero}, for all theoretical tree $\tree^{\star}$
	\begin{align}
	\lim\limits_{K \to \infty} \sum_{j \in \mathcal{J}} \mdi_{\tree_K^{\star}}(X^{(j)}) = \mathds{V} \big[ m_{\mathcal{J}}(\bX^{(\mathcal{J})}) \big], \label{proof_th_mdi_indep1}
	\end{align}
	and similarly, 
	\begin{align}
	\lim\limits_{K \to \infty} \sum_{j \in \mathcal{J}^c} \mdi_{\tree_K^{\star}}(X^{(j)}) = \mathds{V} \big[ m_{\mathcal{J}^c}(\bX^{(\mathcal{J}^c)}) \big] \label{proof_th_mdi_indep2}
	\end{align}
	This proves the first statement of Theorem~\ref{thm:mdi_interaction}. 
	
	\paragraph{Second statement.} To prove the second statement, we need to prove that the convergence of the population MDI in equations \eqref{proof_th_mdi_indep1} and \eqref{proof_th_mdi_indep2} holds uniformly over the set of all possible theoretical trees. Then, we need to extend the convergence to the empirical MDI computed with empirical CART trees. This can be done following the same reasoning as in the proof of Theorem~\ref{thm:sumImportance_emp}.  
	
\end{proof}

\begin{proof}[Proof of Theorem~\ref{thm:additive_models}]

Theorem~\ref{thm:additive_models} is a direct application of Theorem~\ref{thm:mdi_interaction}, by noting that the assumption 
$$\lim_{k \to \infty} \Delta (m , A_{\tree_k^{\star}}(\bX')) = 0$$ 
in Theorem~\ref{thm:mdi_interaction} results from Lemma~1 in \cite{ScBiVe15}.
\end{proof}

To prove  Theorem~\ref{thm:lin_reg}, we need the following Technical Lemma, whose proof is postponed to the end of this section. 

\begin{techlemme}\label{thm:mdi_splimpification_critere}
	Assume that 
	$$
	Y = m(\bX) + \varepsilon,
	$$
	where $\varepsilon$ is a centred noise of finite variance, independent of $\bX$. Thus, the theoretical splitting criterion in the cell $A \subset [0,1]^d$ satisfies
	\begin{align*}
	L^{\star}_A(\ell,s) & = - \mathds{E}^2[m(\bX)|\bX\in A] + \mathds{P}[\bX\in A_L | \bX \in A] \mathds{E}^2[m(\bX) |\bX\in A_L] \\
	& \qquad + \mathds{P}[\bX\in A_R | \bX \in A] \mathds{E}^2[m(\bX) |\bX\in A_R],
	\end{align*}
	where $A_L = A \cap \{ X^{(\ell)} < s\}$ and $A_R = A \cap \{ X^{(\ell)} \geq s\}$.
\end{techlemme}

\begin{proof}[Proof of Theorem \ref{thm:lin_reg}]
	Let $A = \prod_{\ell=1}^d [a_{\ell}, b_{\ell}] \subset [0,1]$ be a generic cell. According to Lemma~\ref{thm:mdi_splimpification_critere}, the splitting criterion writes 
	\begin{align*}
	L^{\star}_A(j,s) & = - \big( \mathds{E} [\alpha_j X^{(j)} |X \in A]\big)^2 
	+ \frac{s-a_j}{b_j - a_j} \big( \mathds{E} [\alpha_j X^{(j)} |X \in A, X^{(j)}\leq s] \big)^2 \\
	& \quad + \frac{b_j - s}{b_j - a_j} \big( \mathds{E} [\alpha_j X^{(j)} |X \in A, X^{(j)}> s] \big)^2\\
	& = - \alpha_j^2 \Big(\frac{a_j + b_j}{2}\Big)^2 
	+  \alpha_j^2 \Big(\frac{s-a_j}{b_j - a_j}\Big)\Big(\frac{a_j+s}{2}\Big)^2
	+ \alpha_j^2 \Big(\frac{b_j - s}{b_j - a_j}\Big)\Big(\frac{b_j+s}{2}\Big)^2\\
	& = - \alpha_j^2 \Big(\frac{a_j + b_j}{2}\Big)^2 
	+  \alpha_j^2 \Big(s^2 - a_j^2\Big)\Big(\frac{a_j+s}{4(b_j-a_j)}\Big)
	+ \alpha_j^2 \Big((b_j^2 - s^2)\Big)\Big(\frac{b_j+s}{4(b_j-a_j)}\Big)\\
	& = \Big(\frac{\alpha_j^2}{4}\Big) s^2 - \alpha_j^2 \Big(\frac{a_j + b_j}{4}\Big)s + \frac{\alpha_j^2 a_j b_j}{4}. 
	\end{align*}
	For any $j$, the function $s \mapsto L^{\star}_A(j,s)$ reaches its maximum at $s^{\star} = (a_j+b_j)/2$. For this particular value, we have
	\begin{align*}
	L^{\star}_A(j,s^{\star}) & = \Big(\frac{\alpha_j^2}{4}\Big) \Big(\frac{a_j+b_j}{2}\Big)^2 - \alpha_j^2 \Big(\frac{a_j + b_j}{4}\Big)\Big(\frac{a_j+b_j}{2}\Big) + \frac{\alpha_j^2 a_j b_j}{4} \nonumber \\
	& = \frac{\alpha_j^2}{4} \Big[ \frac{b_j-a_j}{2}\Big]^2.
	\end{align*}
	Thus the coordinate chosen for splitting in cell $A$ must satisfy
	\begin{align}
	j^{\star} \in \argmax\limits_{\ell \in \{1, \hdots , d\}} \alpha_{\ell}^2 (b_{\ell}-a_{\ell})^2\, . \label{eq:linear_2}
	\end{align}
	Since the chosen coordinate satisfies equation (\ref{eq:linear_2}) and the range of the cell along the chosen side is halved after the split, the diameter of any cell tends to zero, which proves that the theoretical tree is consistent. Hence Theorem \ref{thm:mdi_interaction} can be applied and leads to 
	\begin{align*}
	\lim_{k \to \infty} \mdi_{\tree_k^{\star}}(X^{(j)}) = \alpha_j^2 \mathds{V}[X^{(j)}] = \frac{\alpha_j^2}{12},
	\end{align*}
	since $X^{(j)} \sim \mathcal{U}([0,1])$. 
	The convergence of the empirical MDI directly results from Theorem~\ref{thm:mdi_interaction}.
\end{proof}

\begin{proof}[Proof of Theorem \ref{thm:dim2_product_mdibis}]
	Let $A = \prod_{\ell=1}^d [a_{\ell}, b_{\ell}] \subset [0,1]$ be a generic cell. According to Technical  Lemma~\ref{thm:mdi_splimpification_critere}, the splitting criterion writes
	\begin{align*}
	L^{\star}_A(\ell,s) & = \mathds{V}[Y|\bX\in A] - \mathds{P}[\bX\in A_L | \bX \in A] \mathds{V}[Y |\bX\in A_L] \\
	& \qquad - \mathds{P}[\bX\in A_R | \bX \in A] \mathds{V}[Y |\bX\in A_R]\\
	& = - \mathds{E}^2[m(\bX)|\bX\in A] + \mathds{P}[\bX\in A_L | \bX \in A] \mathds{E}^2[m(\bX) |\bX\in A_L] \\
	& \qquad + \mathds{P}[\bX\in A_R | \bX \in A] \mathds{E}^2[m(\bX) |\bX\in A_R]\\
	& = - \mathds{E}^2\Big[2^k \alpha \prod_{j=1}^k X_j|\bX\in A\Big] + \frac{s-a_{\ell}}{b_{\ell} - a_{\ell}} \mathds{E}^2\Big[2^k \alpha \prod_{j=1}^k X_j |\bX\in A, X_{\ell} <s\Big] \\
	& \qquad + \frac{b_{\ell} - s}{b_{\ell} - a_{\ell}} \mathds{E}^2\Big[2^k \alpha \prod_{j=1}^k X_j |\bX\in A, X_{\ell} \geq s\Big]\\
	& = \frac{\alpha^2}{b_{\ell} - a_{\ell}} \prod_{j \neq \ell} (a_j + b_j)^2 \Big[
	- (a_{\ell} + b_{\ell})^2 (b_{\ell}-a_{\ell}) + (s -a_{\ell})(s+a_{\ell})^2 \\
	& \qquad + (b_{\ell}-s)(s+b_{\ell})^2 \Big]\\
	& = \alpha^2  \Big(-s^2 + (a_{\ell} + b_{\ell}) - a_{\ell} b_{\ell} \Big) \prod_{j \neq \ell} (a_j + b_j)^2.
	\end{align*}
	For any $\ell$, the function $s \mapsto L^{\star}_A(\ell,s)$ reaches its maximum at $s^{\star} = (a_j+b_j)/2$. For this particular value, we have
	\begin{align}
	L^{\star}_A(\ell,s^{\star}) & =  \frac{\alpha^2}{4} (b_{\ell} - a_{\ell})^2 \prod_{j \neq \ell} (a_j + b_j)^2. \label{eq:multiplicative_model}
	\end{align}
	The coordinate $\ell^{\star}$ is chosen for splitting if $L^{\star}_A(\ell^{\star},s^{\star}) = \max\limits_{1 \leq k \leq d} L^{\star}_A(k,s^{\star})$, that is, for all $k$,
	\begin{align*}
	% (b_{\ell^{\star}} - a_{\ell^{\star}})^2 \prod_{j \neq \ell^{\star}} (a_j + b_j)^2 & \geq  (b_{k} - a_{k})^2 \prod_{j \neq k} (a_j + b_j)^2\\
	    \frac{b_{\ell^{\star}} - a_{\ell^{\star}}}{a_{\ell^{\star}} + b_{\ell^{\star}}}   & \geq \frac{b_{k} - a_{k}}{a_{k} + b_{k}}.
	%\Leftrightarrow &   \Big(\frac{a_k + b_k}{2}\Big) (b_{\ell} - a_{\ell}) \geq  (b_{k} - a_{k}) \Big(\frac{a_{\ell} + b_{\ell}}{2}\Big).\\
	\end{align*}
	The variance reduction induced by the split $(\ell^{\star}, s^{\star})$ in the cell $A = \prod_{j=1}^d [a_j, b_j]$ is equal to 
	\begin{align}
	\frac{\alpha_{\ell^{\star}}^2}{4} (b_{\ell^{\star}} - a_{\ell^{\star}})^2 \prod_{j \neq \ell^{\star}} (a_j + b_j)^2. \label{eq:variance_reduction_mult_model}
	\end{align}
	
	Now, fix $\bx \in (0,1]^d$. We want to prove that $\textrm{diam}(A_k(\bx))$ tends to zero as 
	as $k$ tends to infinity. For any cell $A$, and all $\ell \in \{1, \hdots , d\}$, according to equation (\ref{eq:multiplicative_model}), we have 
	\begin{align}
	\frac{\alpha^2}{4} (b_{\ell} - a_{\ell})^2 \prod_{j \neq \ell} x_j^2. \leq L^{\star}_A(\ell,s^{\star}) \leq  \frac{\alpha^2}{4} (b_{\ell} - a_{\ell})^2 . \label{eq:multiplicative2} 
	\end{align}
	For all $k$, let $A_k(\bx) = \prod_{j=1}^d [a_{j,k}, b_{j,k}]$ be the cell containing $\bx$ in the theoretical tree $\tree^{\star}_k$. Let $u^{(j)}_k = b_{j,k} - a_{j,k}$. Since, for all $j$, $(u^{(j)})$ is a positive decreasing sequence, we obtain that $(u^{(j)})$ converges toward $\gamma_j \geq 0$. Let $\Gamma = \{j, \gamma_j >0\}$ and assume that $\Gamma \neq \emptyset$. For all $j$, 
	\begin{align*}
	\left\lbrace
	\begin{array}{ll}
	u^{(j)}_{k+1} = u^{(j)}_k & \textrm{if the split is not performed along the $j$th coordinate} \\
	u^{(j)}_{k+1} = \frac{u^{(j)}_k}{2} & \textrm{otherwise}. 
	\end{array}
	\right..
	\end{align*}
	Hence, there exists $K$ such that for all $k>K$, and all $j \in \mathcal{J}$, 
	\begin{align}
	u^{(j)}_k = \gamma_j. \label{eq:limite_diff_cell_side}
	\end{align}
	In other words, after level $k$, all splits are performed along coordinates belonging to $\mathcal{J}^c$. Now, for all $j \in \mathcal{J}^c$, $\lim_{k \to \infty} u^{(j)}_k = 0$, and thus, according to equation (\ref{eq:multiplicative2}),
	$$\lim\limits_{k \to \infty} \max\limits_{j \in \mathcal{J}^c } L^{\star}_{A_k(\bx)}(j,s^{\star}) = 0.$$
	Therefore, there exists $K_1>K$, such that, for all $k>K_1$, 
	\begin{align*}
	\max\limits_{j \in \mathcal{J}^c } L^{\star}_{A_k(\bx)}(j,s^{\star}) < \min\limits_{\ell \in \mathcal{J} } \frac{\alpha^2}{8} \gamma_{\ell}^2 \prod_{j \neq \ell} x_j^2.
	\end{align*}
	Using again Equation (\ref{eq:multiplicative2}), we deduce that, for all $k>K_1$, 
	\begin{align*}
	\max\limits_{j \in \mathcal{J}^c } L^{\star}_{A_k(\bx)}(j,s^{\star}) < \min\limits_{\ell \in \mathcal{J} } L^{\star}_{A_k(\bx)}(\ell,s^{\star}),
	\end{align*}
	thus the next split is performed along $j \in \mathcal{J}$ which is absurd. Hence, either $\mathcal{J} = \emptyset$ or $\mathcal{J}^c = \emptyset$. If $\mathcal{J}^c = \emptyset$, given the expression of the variance reduction in equation \eqref{eq:variance_reduction_mult_model}, an additional split must occur after rank $K$, with an additional variance decreasing, which is forbidden by equation \eqref{eq:limite_diff_cell_side}. Finally, $\mathcal{J} = \emptyset$ and consequently, for all $j \in \{1, \hdots, d\}$,  $\lim\limits_{k \to \infty} u^{(j)}_k=0$, which implies, for all $\bx \in (0,1]^d$, 
	$$\lim\limits_{k \to \infty} \textrm{diam}(A_k(\bx)) = 0.$$
	Since $\bX$ is uniformly distributed over $[0,1]^d$, and $m$ is continuous, we have, almost surely,
	$$\lim\limits_{k \to \infty} \Delta (m , A_{\tree_k^{\star}}(\bX)) = 0.$$
	A direct application of Theorem~\ref{thm:sumImportance} yields 
	\begin{align*}
	\sum_{j=1}^d \mdi_{\tree_{\infty}^{\star}}(X^{(j)}) = \mathds{V}[m(\bX)].
	\end{align*}
	%Recall that $m(\bX) = 2^d \alpha \prod_{j=1}^d X^{(j)}$.
	A direct calculation shows that $\mathds{E}[m(\bX)] = \alpha$,
	and $\mathds{E}[m(\bX)^2] = \alpha^2 \left(4/3 \right)^d$,
	leading directly to the conclusion. 
	%and let $j^{\star}$ be the index such that
	%\begin{align*}
	%\mathcal{L}_{A_{\infty}}(j^{\star}, s^{\star}) = \max_{1 \leq j \leq d } \mathcal{L}_{A_{\infty}}(j, s^{\star}) = c.
	%\end{align*}
	The last statement is a direct application of Theorem \ref{thm:sumImportance_emp}.
\end{proof}

\begin{proof}[Proof of Lemma~\ref{lem:diff_imp}]
	We prove the result in dimension two. The generalization in any dimension will directly follow. According to Theorem~\ref{thm:dim2_product_mdibis}, in the root cell $[0,1]^2$, the variable $\ell^{\star}$ selected for splitting must satisfy 
	\begin{align*}
	\ell^{\star} \in \argmax_{k \in \{1,2\}} \Big(\frac{a_k + b_k}{2}\Big) (b_{\ell} - a_{\ell}),  
	\end{align*}
	where the right-hand term is equal to $1/2$ for $k=1,2$. Therefore, the selected variable can be $X^{(1)}$ or $X^{(2)}$. The corresponding decrease in variance is equal to $\alpha^2/4$. 
	
	\paragraph{First theoretical tree.} Consider the theoretical tree $\tree^{\star}_{2,1}$ whose construction is stopped at level two, where the first split is performed along $X^{(1)}$. The same calculation as above shows that splits must occur along the second variable which gives a variance decreasing of $\alpha^2/16$ in the left cell, and $9\alpha^2/16$ in the 
	right cell. The corresponding partition is shown on the left of Figure~\ref{figure2}. The variable importances for the tree $\tree^{\star}_{2,1}$ are
	\begin{align*}
	\mdi_{\tree^{\star}_{2,1}}(X^{(1)}) = \alpha^2/4
	\end{align*}
	and
	\begin{align*}
	\mdi_{\tree^{\star}_{2,1}}(X^{(2)}) = \frac{1}{2} \left( \alpha^2/16 + 9 \alpha^2/16 \right) = 5 \alpha^2/16
	\end{align*}
	
	\paragraph{Second theoretical tree.} Consider the theoretical tree    $\tree^{\star}_{2,2}$ whose construction is stopped at level two, where the first split is made along the second variable. Then the same reasoning as above shows that the second split must be made along the second variable, the resulting partition is displayed on the right of Figure ~\ref{figure2}. The variable importances for the tree $\tree^{\star}_{2,2}$ are thus 
	\begin{align*}
	\mdi_{\tree^{\star}_{2,2}}(X^{(1)}) = 5 \alpha^2/16
	\end{align*}
	and
	\begin{align*}
	\mdi_{\tree^{\star}_{2,2}}(X^{(2)}) =  \alpha^2/4
	\end{align*}
	
	\paragraph{Conclusion} The partitions of the two theoretical trees $\tree^{\star}_{2,1}$ and $\tree^{\star}_{2,2}$ defined above are the same but variable importances are not. Since the partitions are the same, one can choose any splitting strategy for the following splits and use this exact same strategy in both $\tree^{\star}_{2,1}$ and $\tree^{\star}_{2,2}$. Since the splits occurring at level $k \geq 2$ are set to be the same, we have
	\begin{align*}
	\mdi_{\mathcal{T}_{k,2}}(X^{(1)}) - \mdi_{\mathcal{T}_{k,1}}(X^{(1)}) = 5 \alpha^2/16 - \alpha^2/4 = \alpha^2/16.
	\end{align*}
	
\end{proof}

%\begin{proof}[Proof of Theorem~\ref{thm:generic_regression_model}]
%	
%Model~\ref{ass:add_mult_models} is a sum of functions in which all input variables are independent. Therefore, Theorem~\ref{thm:mdi_interaction} applies. According to Theorem~\ref{ass:mult_models}, we know that, any theoretical tree built on each function in the sum of Model~\ref{ass:add_mult_models} is consistent, thus theoretical trees built on the original distribution is consistent according to statement $(ii)$ in  Theorem~\ref{thm:mdi_interaction}. Consequently, statement $(iii)$ of Theorem~\ref{thm:mdi_interaction} is valid and yields 
%\begin{align*}
%\lim\limits_{k \to \infty} \sum_{j: \delta_{\ell, j}=1} \mdi_{\tree_k^{\star}}(X^{(j)})  = \alpha_{\ell}^2 \Bigg( \Big(\frac{4}{3}\Big)^{D_{\ell}} - 1\Bigg).
%\end{align*}
%	
%\end{proof}

\begin{proof}[Proof of Lemma \ref{lemma:correlationX1X2}]
	We have $\bX \sim \mathcal{U}\left( \cup_{j=1}^{2^{\beta}} \Big[ \frac{j-1}{2^{\beta}}, \frac{j}{2^{\beta}}\Big)^2 \right) = \unifbeta.$ Therefore, 
	\begin{align*}
	\textrm{Cov}(X^{(1)}, X^{(2)}) & = \mathds{E}\Big[X^{(1)} X^{(2)} \sum_{j=1}^{2^{\beta}} \mathds{1}_{\bX \in  [ \frac{j-1}{2^{\beta}}, \frac{j}{2^{\beta}})^2}\Big] - \frac{1}{4}\\
	& = \sum_{j=1}^{2^{\beta}} \mathds{E} \Big[ \mathds{1}_{\bX \in  [ \frac{j-1}{2^{\beta}}, \frac{j}{2^{\beta}})^2} \mathds{E}\Big[X^{(1)} X^{(2)} \Big| \bX \in  \big[ \frac{j-1}{2^{\beta}}, \frac{j}{2^{\beta}}\big)^2 \Big]  \Big]   - \frac{1}{4},
	\end{align*}
	with
	\begin{align*}
	\mathds{E}\Big[X^{(1)} X^{(2)} \Big| \bX \in  \big[ \frac{j-1}{2^{\beta}}, \frac{j}{2^{\beta}}\big)^2 \Big] = \frac{1}{2^{2\beta}} \Big(  j - \frac{1}{2}\Big)^2,
	\end{align*}
	which leads to
	\begin{align*}
	\textrm{Cov}(X^{(1)}, X^{(2)}) & = \sum_{j=1}^{2^{\beta}} \frac{1}{2^{3\beta}} \Big( j - \frac{1}{2} \Big)^2 - \frac{1}{4}\\
	& =  \frac{1}{2^{3\beta+2}} \sum_{j=1}^{2^{\beta}}   (2j -1)^2 - \frac{1}{4}.
	\end{align*}
	Since we know that 
	\begin{align*}
	\sum_{j=1}^{2^{\beta}} (2j-1)^2 = \frac{{2^{\beta}} ({2^{2\beta+2}} - 1)}{3} ,
	\end{align*}
	the covariance can be expressed as
	\begin{align*}
	\textrm{Cov}(X^{(1)}, X^{(2)}) 
	%& = \frac{1}{4} \Big(\frac{1}{2^{\beta}}\Big)^3 \frac{1}{3} 2^{\beta}  (4 . 2^{2{\beta}} - 1) - \frac{1}{4}\\
	%& = \frac{1}{3. 2^{2{\beta}+2}} (2^{2{\beta}+2} - 1) - \frac{1}{4} \\
	& = \frac{1}{12} - \frac{1}{3. 2^{2{\beta}+2}}.
	\end{align*}
	Since the marginals are uniform on $[0,1]$, we have $\mathds{V}[X^{(1)}]= \mathds{V}[X^{(2)}]= 1/12$, which yields,
	\begin{align*}
	\textrm{Corr}(X^{(1)}, X^{(2)}) 
	%& = \frac{\textrm{Cov}(X_1, X_2)}{\sqrt{ \mathds{V}[X_1]  \mathds{V}[X_2] }}\\
	%& = 12 \Big( \frac{1}{12} - \frac{1}{3. 2^{2{\beta}+2}}\Big)\\
	& = 1 - \frac{1}{4^{\beta}}.
	\end{align*}
\end{proof}

\begin{proof}[Proof of Proposition~\ref{prop:cor_vs_uncor_input}]

The first statement of Theorem~\ref{thm:mdi_interaction} remains valid even when the density of $\bX$ is not lower bounded by some positive constant. Besides, using Technical Lemma~\ref{thm:mdi_splimpification_critere}, the splitting criterion for Model~\ref{ass:last_theorem_corr}, computed for a split along the first coordinate writes 
\begin{align}
L^{\star}(1,s) & = - \mathds{E}^2 [X^{(1)} + X^{(2)}| (X^{(1)},X^{(2)}) \in [0,1]^2 ] \nonumber \\
& \quad + s \mathds{E}^2[X^{(1)} + X^{(2)}| (X^{(1)},X^{(2)}) \in [0,s]\times [0,1]] \nonumber \\
& \quad + (1-s) \mathds{E}^2[X^{(1)} + X^{(2)}| (X^{(1)},X^{(2)}) \in [s,1]\times [0,1]] \nonumber \\
& = - 1 + s \Big( \frac{s}{2} + \mathds{E}[X^{(2)}| (X^{(1)},X^{(2)}) \in [0,s]\times [0,1]]\Big)^2 \nonumber \\
& \quad + (1-s) \Big( \frac{1+s}{2} + \mathds{E}[X^{(2)} | (X^{(1)},X^{(2)}) \in [s,1]\times [0,1]] \Big)^2. \label{eq:split_crit_corr}
\end{align}
Let us fix a split at position $s$ along variable $X^{(1)}$ and let $\ell = \lfloor 2^{\beta} s \rfloor$. The conditional density of $X^{(2)}$ given that $(X^{(1)},X^{(2)}) \in [0,s] \times [0,1]$ is 
\begin{align*}
f_{X^{(2)} | X^{(1)} \in [0,s]} (x_2) = C \Big( \frac{1}{2^{\beta}} \mathds{1}_{2^{\beta} x_2 \leq \ell} 
+ \Big(s - \frac{\ell}{2^{\beta}} \Big) \mathds{1}_{\ell < 2^{\beta} x_2 \leq \ell +1}\Big).
\end{align*}
where, in order for $f_{X_2 | X_1 \in [0,s]}$ to be a density, $C$ must satisfy
\begin{align*}
& C \Big( \frac{\ell}{4^{\beta}} + \Big( s - \frac{\ell}{2^{\beta}}\Big)\frac{1}{2^{\beta}}\Big) = 1,
\end{align*}
that is $ C = 2^{\beta}/s$. Hence, 
\begin{align*}
f_{X^{(2)} | X^{(1)} \in [0,s]} (x_2) = \frac{2^{\beta}}{s} \Big( \frac{1}{2^{\beta}} \mathds{1}_{2^{\beta} x_2 \leq \ell} 
+ \Big(s - \frac{\ell}{2^{\beta}} \Big) \mathds{1}_{\ell < 2^{\beta} x_2 < \ell +1}\Big),
\end{align*}
and similarly, 
\begin{align*}
f_{X^{(2)} | X^{(1)} \in [s,1]}(x_2) = \frac{2^{\beta}}{1-s} \Big( \Big( \frac{\ell + 1}{2^{\beta}} - s \Big) \mathds{1}_{\ell < 2^{\beta} x_2 < \ell +1} 
+ \frac{1}{2^{\beta}} \mathds{1}_{2^{\beta} x_2 \geq \ell +1}\Big).
\end{align*}
To make explicit the splitting criterion in equation (\ref{eq:split_crit_corr}), we need to compute, on one hand, 
\begin{align}
& \mathds{E}[X^{(1)} + X^{(2)} | (X^{(1)},X^{(2)}) \in [0,s] \times [0,1] ] \nonumber \\
& = \frac{s}{2} + \frac{1}{s} \int_0^{\ell /2^{\beta}} x_2 \textrm{d}x_2
+ \frac{2^{\beta}}{s} \Big(s - \frac{\ell }{2^{\beta}} \Big) \int_{\ell /2^{\beta}}^{(\ell +1)/2^{\beta}} x_2 \textrm{d}x_2 \nonumber \\
& = \frac{s}{2} + \frac{1}{2s} \Big( \frac{\ell }{2^{\beta}}\Big)^2 + \frac{2^{\beta}}{2s} \Big( s - \frac{\ell }{2^{\beta}}\Big) \Big( \Big( \frac{\ell  +1}{2^{\beta}}\Big)^2 - \Big(\frac{\ell }{2^{\beta}}\Big)^2\Big) \nonumber \\
& = \frac{s}{2} + \frac{1}{2s} \Big( \frac{\ell }{2^{\beta}}\Big)^2 + \frac{2 \ell +1}{2^{{\beta}+1}s} \Big( s - \frac{\ell }{2^{\beta}}\Big), \label{eq:part1_split_corr}
\end{align}
and on the other hand, 
\begin{align}
& \mathds{E}[X^{(1)} + X^{(2)} | (X^{(1)},X^{(2)})  \in [s,1]\times [0,1] ] \nonumber \\
& = \frac{1+s}{2} + 
\frac{2^{\beta}}{1-s} \Big( \frac{\ell+1}{2^{\beta}} - s \Big)
\int_{\ell/2^{\beta}}^{(\ell+1)/2^{\beta}} x_2 \textrm{d}x_2 
+ \frac{2^{\beta}}{1-s} \frac{1}{2^{\beta}} \int_{(\ell+1)/2^{\beta}}^{1} x_2 \textrm{d}x_2 \nonumber \\
& = \frac{1+s}{2} + \frac{2^{\beta}}{2(1-s)} \Big( \frac{\ell+1}{2^{\beta}} -s\Big) \Bigg( \Big( \frac{\ell+1}{2^{\beta}}\Big)^2 - \Big( \frac{\ell}{2^{\beta}}\Big)^2\Bigg) + \frac{1}{2(1-s)} \Bigg( 1 - \Big( \frac{\ell +1}{2^{\beta}}\Big)^2\Bigg) \nonumber\\
& = \frac{1+s}{2} - \frac{(2 \ell +1)s}{2^{{\beta}+1} (1-s)} + \frac{1}{2(1-s)} \Big( \frac{\ell (\ell+1)}{2^{2{\beta}}}+1\Big).	
\label{eq:part2_split_corr}
\end{align}
Injecting equations (\ref{eq:part1_split_corr}) and (\ref{eq:part2_split_corr}) into (\ref{eq:split_crit_corr}), we get
\begin{align}
L^{\star}(1,s) & = -1 + s 
\Big( \frac{s}{2} + \frac{2\ell+1}{2^{{\beta}+1}} - \frac{\ell^2 + \ell}{2^{2{\beta}+1} s } \Big)^2 \nonumber \\
& \qquad + (1-s) \Bigg( \frac{1+s}{2} - \frac{(2 \ell +1)s}{2^{{\beta}+1} (1-s)} + \frac{1}{2(1-s)} \Big( \frac{\ell (\ell+1)}{2^{2{\beta}}}+1\Big)   \Bigg)^2. \label{eq:split_crit_corr2}
\end{align}
The second term in equation (\ref{eq:split_crit_corr2}) can be expressed as
\begin{align*}
\Big( \frac{s}{2} + \frac{2\ell+1}{2^{{\beta}+1}} - \frac{\ell^2 + \ell}{2^{2{\beta}+1} s } \Big)^2 &=  \Big( \frac{2 \ell +1}{2^{{\beta}+1}}\Big)^2 + \frac{s^2}{4} 
+ \Big(\frac{\ell^2 + \ell}{2^{2{\beta}+1} s}\Big)^2
- \frac{\ell^2 + \ell}{2^{2{\beta}+1}}\\
& \quad + s \Big(  \frac{2 \ell +1}{2^{{\beta}+1}}\Big)
- \Big(\frac{2 \ell +1}{2^{{\beta}+1}}\Big) \Big(\frac{\ell^2 + \ell}{2^{2{\beta}+1} s}\Big).
\end{align*}
and third term in equation (\ref{eq:split_crit_corr2}) as
\begin{align*}
& \Bigg( \frac{1+s}{2} - \frac{(2 \ell +1)s}{2^{{\beta}+1} (1-s)} + \frac{1}{2(1-s)} \Big( \frac{\ell (\ell+1)}{2^{2{\beta}}}+1\Big)   \Bigg)^2 \\
= ~& 
\frac{1}{4} + \frac{s}{2}
+ \frac{s^2}{4} +  \frac{1}{2} \Big(\frac{\ell (\ell+1)}{2^{2{\beta}}} +1\Big) \frac{1}{1-s}  + \frac{s}{1-s}\Bigg( \frac{1}{2} \Big( \frac{\ell(\ell+1)}{2^{2{\beta}}} +1\Big) - \frac{ 2 \ell +1}{2^{{\beta}+1}}\Bigg)\\
&  
- \frac{(2\ell+1)s^2}{2^{{\beta}+1} (1-s)}
+ \frac{1}{4(1-s)^2} \Big(\frac{\ell (\ell+1)}{2^{2{\beta}}}+1\Big)^2
- \Big( \frac{\ell (\ell+1)}{2^{2{\beta}}} +1\Big) \Big(\frac{2\ell+1}{2^{{\beta}+1}}\Big) \frac{s}{(1-s)^2}\\
&  + \Big(\frac{2\ell+1}{2^{{\beta}+1}}\Big)^2 \frac{s^2}{(1-s)^2},
\end{align*}
which, after several simplifications, leads to
\begin{align*}
L^{\star}(1,s) & = - \frac{1}{4} s^2 + s\Big( \frac{3}{4} + \frac{4 \ell^2 + 4 \ell +1}{2^{2{\beta}+2}} - \Big( \frac{2 \ell +1}{2^{{\beta}+1}}\Big)\Big) 
- \frac{3}{4} + \frac{1}{2}\Big( \frac{\ell (\ell+1)}{2^{2{\beta}}} +1 \Big)\\
& \qquad - \Big( \frac{2 \ell +1}{2^{{\beta}+1}}\Big) \Big( \frac{\ell^2 + \ell}{2^{2{\beta}}}\Big) 
+ \frac{1}{s} \Big( \frac{\ell^2 + \ell }{2^{2{\beta}+1}}\Big)^2 
+ \frac{1}{4(1-s)} \Big( \frac{\ell (\ell +1)}{2^{2{\beta}}} +1 \Big)^2\\
& \qquad  + \frac{s^2}{1-s} \Big( \frac{2\ell +1}{2^{{\beta}+1}}\Big)^2
- \frac{s}{1-s} \Big( \frac{2 \ell +1}{2^{{\beta}+1}} \Big) \Big( \frac{\ell (\ell +1)}{2^{2{\beta}}} +1 \Big).
\end{align*}
Thus we have
\begin{align}
\frac{\partial L^{\star}(1,s)}{\partial s} & =
- \frac{1}{2} s + \frac{3}{4} + \frac{4 \ell^2 + 4 \ell +1}{2^{2{\beta}+2}} - \Big( \frac{2 \ell +1}{2^{{\beta}+1}} \Big) 
- \frac{1}{s^2} \Big( \frac{\ell^2 + \ell}{2^{2{\beta}+1}}\Big)^2 \nonumber \\
& \qquad + \frac{1}{(1-s)^2} \Bigg( \frac{1}{4} \Big( \frac{\ell (\ell+1)}{2^{2{\beta}}} +1\Big)^2 - \Big( \frac{2 \ell +1}{2^{{\beta}+1}}\Big) \Big(\frac{\ell (\ell+1)}{2^{2{\beta}}} +1 \Big)\Bigg) \nonumber \\
& \qquad + \frac{s(2-s)}{(1-s)^2} \Big(\frac{2 \ell +1}{2^{{\beta}+1}}\Big)^2, \label{eq:diff_criterion}
\end{align}
with
\begin{align*}
\Big| \frac{\partial L^{\star}(1,s)}{\partial s} \Big| & \leq 6 + 12 \times 2^{2{\beta}}.
\end{align*}
We know that the function $s \mapsto L^{\star}(1,s)$  is symmetric around $1/2$ and we  conjecture that it is increasing on $[0,1/2]$. Note that, we have $L^{\star}(1,1/2) = 1/4$.	
For all $\beta$, let $m_{\beta} = L^{\star}(1, 1/2 - 1/2^{\beta})$. Setting a regular grid of size 
\begin{align*}
\varepsilon = \frac{\frac{1}{4} - m_{\beta}}{\Big\| \frac{\partial L^{\star}(1,s)}{\partial s} \Big\|_{\infty}},
\end{align*}
if all evaluations of $L^{\star}(1,s)$ on the previous grid are lower than $m_{\beta}$ then the maximum of $s \mapsto L^{\star}(1,s)$ is reached in $[1/2 - 1/2^{\beta}, 1/2 + 1/2^{\beta}]$. An numerical optimization with the function \texttt{optimize} of $\texttt{R}$ \citep[see][]{OptimizePackage, RPackage} shows that it is true for ${\beta}=1, \hdots, 5$. Figure~\ref{figure4} displays the splitting criterion for ${\beta} = 0, \hdots, 3$.
\begin{figure}[h!!]
	\begin{minipage}[l]{.5\linewidth}
		\begin{center}
			\includegraphics[scale=0.35]{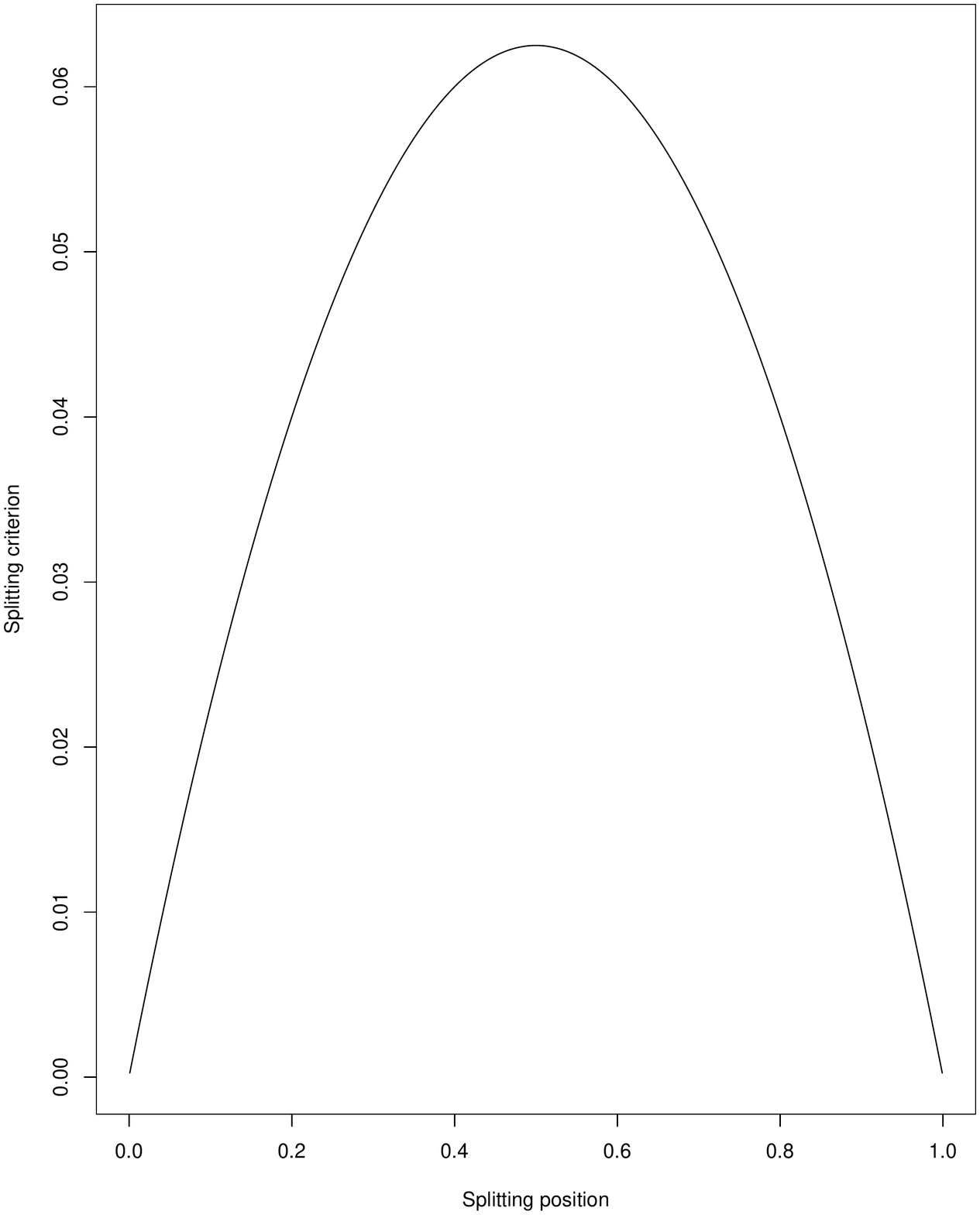}
		\end{center}
		
		\begin{center}
			\includegraphics[scale=0.35]{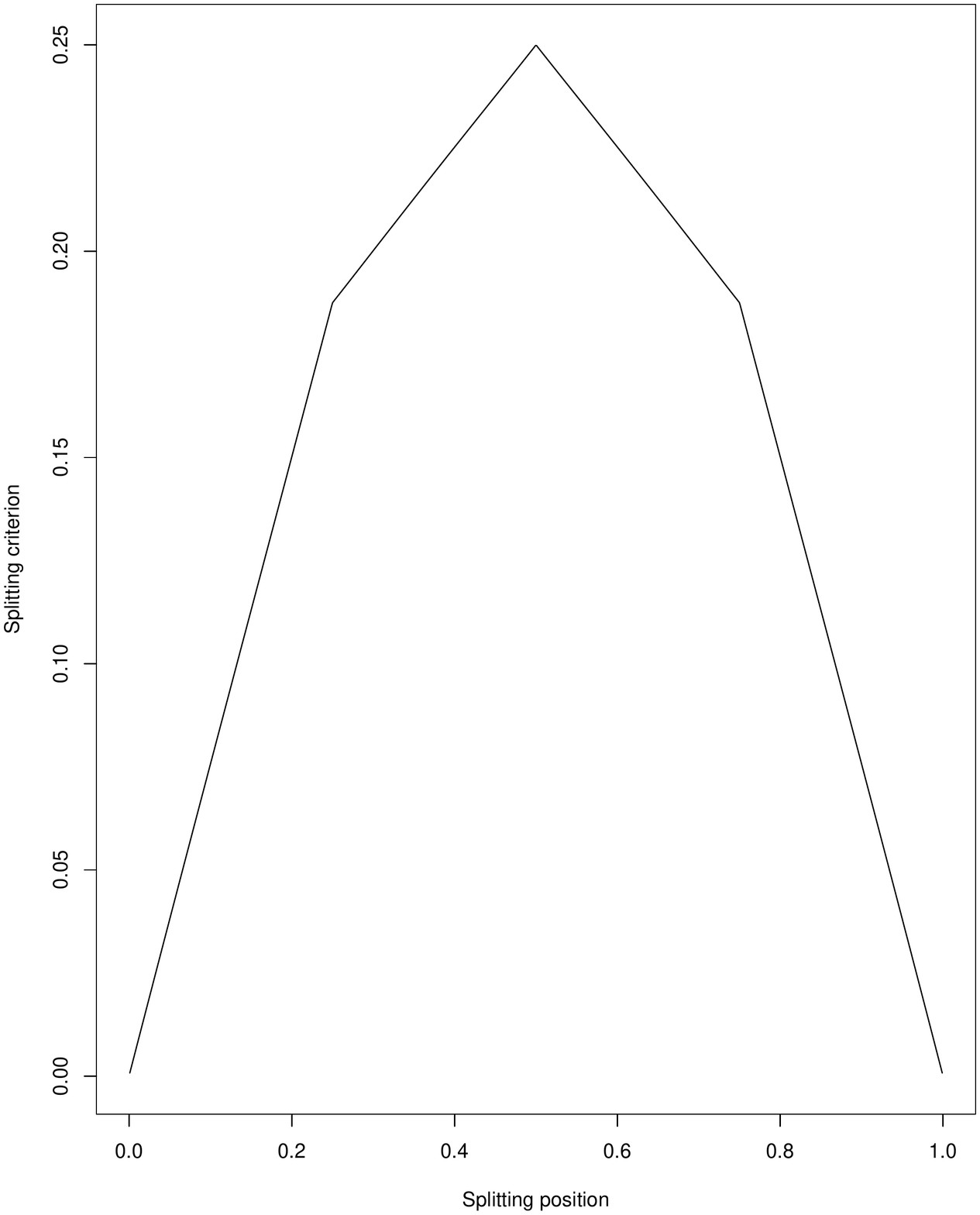}
		\end{center}
	\end{minipage} 
	\hspace{-0.5cm}
	\begin{minipage}[r]{.5\linewidth}
		\begin{center}
			\includegraphics[scale=0.35]{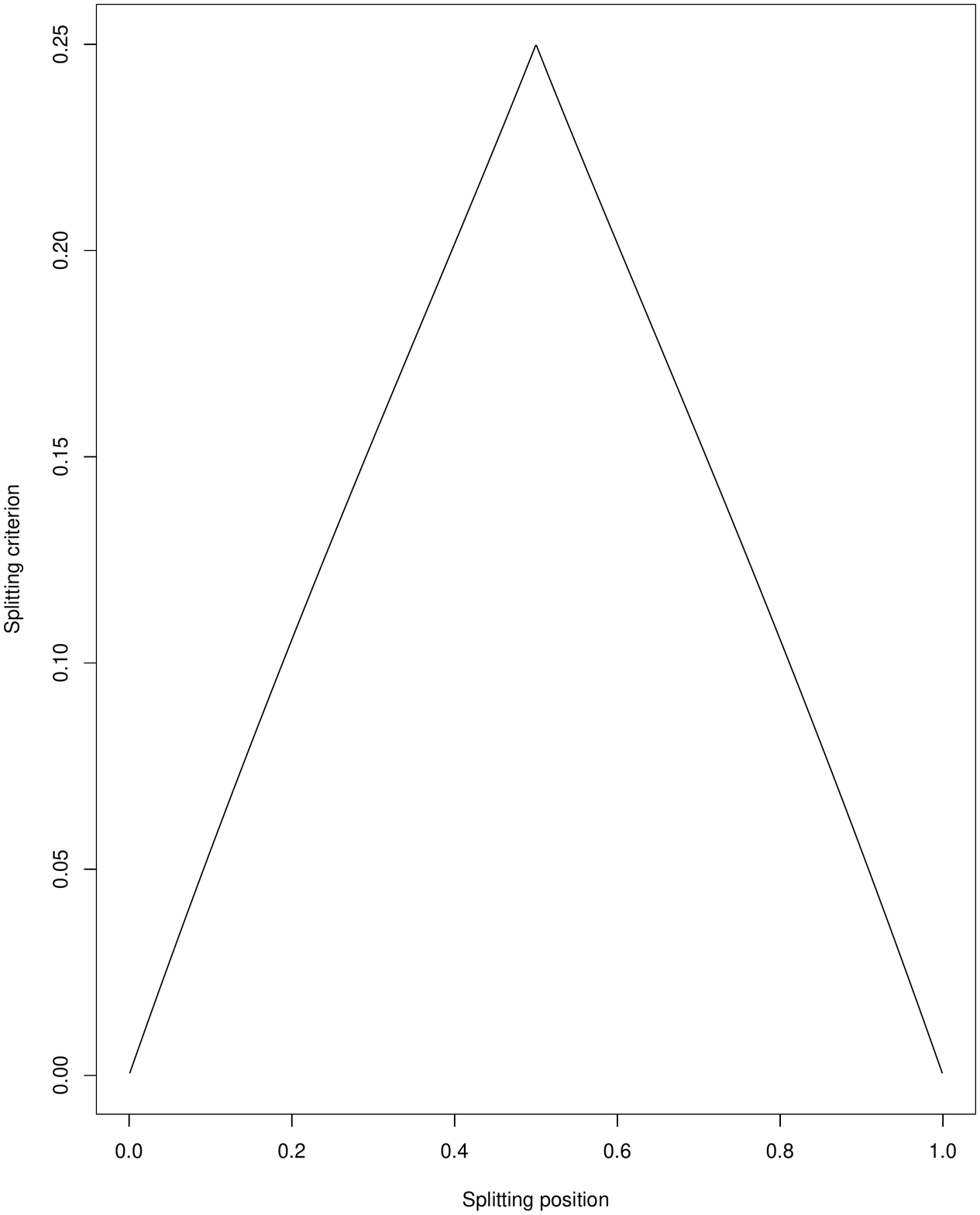}
		\end{center}	
		
		\begin{center}
			\includegraphics[scale=0.35]{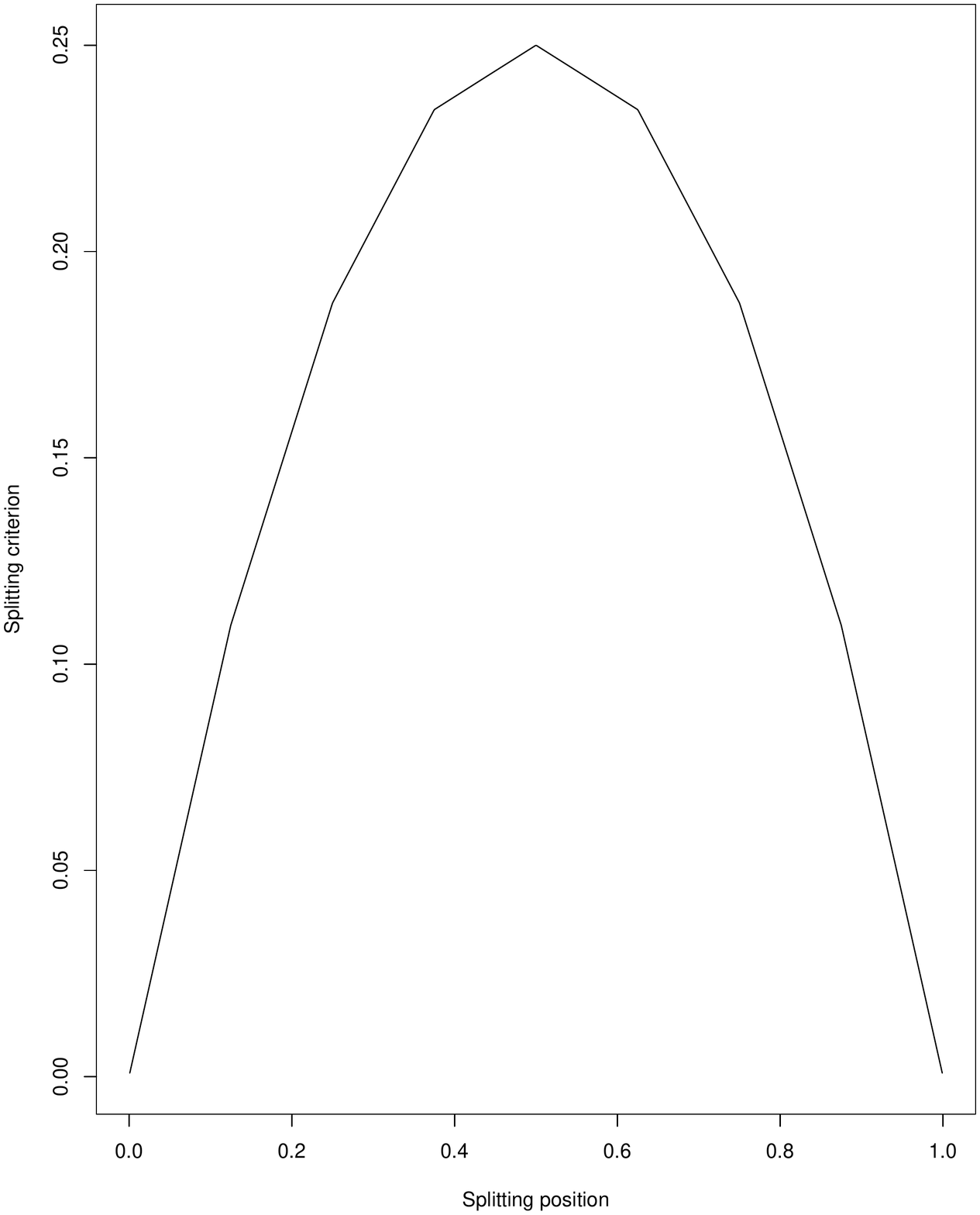}
		\end{center}
	\end{minipage}
	\caption{Illustration of the splitting criterion $L^{\star}(1,s)$ for ${\beta}=0$ (top left), ${\beta}=1$ (top right), ${\beta}=2$ (bottom left), ${\beta}=3$ (bottom right).}
	\label{figure4}
\end{figure}
On the interval $[1/2 - 1/2^{\beta}, 1/2)$, the partial derivative writes
\begin{align*}
\frac{\partial L^{\star}(1,s)}{\partial s} & =
- \frac{1}{2} s + \Big( \frac{3}{4} + \gamma_{\beta} \Big) - \frac{\gamma_{\beta}}{s^2},
\end{align*}
where $$\gamma_{\beta} = - \frac{1}{2^{{\beta}+1}} + \frac{1}{2^{2{\beta}+2}}.$$
Using equation \eqref{eq:diff_criterion} and the fact that $\frac{\partial L^{\star}(1,1)}{\partial s} = 0$, we have
\begin{align}
\frac{\partial L^{\star}(1,s)}{\partial s} & =
- \frac{1}{2} (s-1) \Big(s^2 + \Big(1 - 2 \Big(\frac{3}{4} + \gamma_{\beta}\Big) \Big)s + 2 \Big(\frac{1}{4} + \gamma_{\beta}\Big) \Big), \label{eq:proof_discriminant}
\end{align}
where the discriminant of the second order expression in equation \eqref{eq:proof_discriminant} is, for $\beta \geq 1$, 
\begin{align*}
\Delta & = 4 \gamma_{\beta}^2 - 6 \gamma_{\beta} - \frac{7}{4}\\
& = - \frac{7}{4} + \frac{3}{2^{\beta}} -  \frac{1}{2^{2{\beta}+1}}
- \frac{1}{2^{3{\beta}}} + \frac{1}{2^{4{\beta}+2}}\\
& < 0.	
\end{align*}
Thus the sign of the expression in the right-hand side of equation \eqref{eq:proof_discriminant} does not vary on $[0,1/2) $ and is positive in zero. Thus, the splitting criterion is increasing on  $[1/2 - 1/2^{\beta}, 1/2)$. Since we know with the optimization procedure that the maximum is reached in this interval, it is reached for $s=1/2$. This proves the first statement of Proposition~\ref{prop:cor_vs_uncor_input}.

Regarding the second statement of Proposition~\ref{prop:cor_vs_uncor_input}, note that the above calculations hold also for a split along the second variable. According to Theorem~\ref{thm:lin_reg}, the variance reduction associated to a split performed along $X^{(3)}$ equals $\alpha^2/16$. Thus, for $\beta = 0, \hdots, 5$, a split will occur along the first or second variable if, and only if 
\begin{align*}
& \frac{\alpha^2}{16} < \frac{1}{4},
\end{align*}
that is $|\alpha| < 2$.

\end{proof}

\begin{proof}[Proof of Theorem~\ref{thm:mdi_with_correlation}]
	
	The  distribution of $X^{(3)}$ in each cell is still uniform. Assume that $|\alpha| < 2$, so that the first split is performed on (without loss of generality) the first variable. Then the resulting left cell is $A = [0,1/2] \times [0,1]^2$. For $\beta >1$, the conditional distribution of $(X^{(1)}, X^{(2)})$ on $A$ is given by 
	\begin{align*}
	(X^{(1)}, X^{(2)}) \sim \mathcal{U}\left( \cup_{j=0}^{2^{\beta-1}-1} \Big[ \frac{j}{2^{\beta}}, \frac{j+1}{2^{\beta}}\Big)^2 \right).
	\end{align*}
	On this cell, the regression model still writes 
	$$
	Y = \frac{1}{2} (2X^{(1)}) + \frac{1}{2} (2 X^{(2)}) + \alpha X^{(3)} + \varepsilon,
	$$
	where 
	\begin{align*}
	(2X^{(1)}, 2X^{(2)}) \sim \mathcal{U}\left( \cup_{j=0}^{2^{\beta-1}-1} \Big[ \frac{j}{2^{\beta-1}}, \frac{j+1}{2^{\beta-1}}\Big)^2 \right).
	\end{align*}
	Therefore, the previous reasoning applies and a split is performed on the first or second variable in this cell if, and only if, 
	\begin{align*}
	\frac{1}{16} > \frac{\alpha^2}{16},
	\end{align*}
	that is $|\alpha| < 1$. This holds until each cell has been cut $\beta$ times along the first or second coordinate. In the resulting cells, the distribution of $\bX$ is uniform on its support which is an hyperrectangle, and, according to Theorem~\ref{thm:lin_reg}, the variance of the regression function decreases to zero as a function of the tree level $k$. The previous reasoning shows that the sequence of splits along $\{X^{(1)} ~\textrm{or} ~X^{(2)}\}$ and $X^{(3)}$ is deterministic and depends only of the coefficient $\alpha$.
	The fourth statement, that is the convergence of the empirical MDI directly results from the convergence of theoretical trees proved above, and from a direct application of Theorem~\ref{thm:mdi_interaction}. 

\end{proof}

\begin{proof}[Proof of Lemma~\ref{lem:diff_imp_corr}]
	According to Theorem~\ref{thm:mdi_interaction}, splits along $X^{(3)}$ has no impact on the variance of $X^{(1)} + X^{(2)}$. One can thus assume, without loss of generality that the splits of the first $\beta$ levels are all performed along the first and the second variables. According to the proof of Theorem~\ref{prop:cor_vs_uncor_input}, in each cell of the first $\beta$ levels, the splitting criterion is the same along variable $X^{(1)}$ and $X^{(2)}$. Consider the tree $\tree^{\star}_1$ where all splits in the first $\beta$ levels are made along the first variable. Then, the total decrease of variance associated to all these splits is 
	\begin{align*}
	\mdi_{\tree^{\star}_{\beta,1}}(X^{(1)}) & = \sum_{j=0}^{\beta-1}\frac{1}{4} \Big(\frac{1}{2^j}\Big)^2 
	%& = \frac{1}{4} \frac{1 - (1/4)^{\beta}}{1 - 1/4}\\
	 = \frac{1}{3} - \frac{1}{3} \Big( \frac{1}{4}\Big)^{\beta}.
	\end{align*}
	In each cell that results from the first $\beta$ cuts, the distribution of $(X^{(1)}, X^{(2)})$ is uniform. Then the decrease of variance in all these cells is the same for $X^{(1)}$ and $X^{(2)}$. Hence,
	\begin{align*}
	\lim\limits_{k \to \infty} \Big( \mdi_{\tree^{\star}_{k,1}}(X^{(1)}) - \mdi_{\tree^{\star}_{k,1}}(X^{(2)}) \Big) = \frac{1}{3} - \frac{1}{3} \Big( \frac{1}{4}\Big)^{\beta}.
	\end{align*}
	Let $\tree^{\star}_{2}$ be the exact same tree as $\tree^{\star}_{1}$ but where all the first $\beta$ splits are  performed along the second variable. As before, we have,
	\begin{align*}
	\lim\limits_{k \to \infty} \Big( \mdi_{\tree^{\star}_{k,2}}(X^{(2)}) - \mdi_{\tree^{\star}_{k,2}}(X^{(1)}) \Big) = \frac{1}{3} - \frac{1}{3} \Big( \frac{1}{4}\Big)^{\beta}.
	\end{align*}
	According to Theorem~\ref{prop:cor_vs_uncor_input}, we know that 
	\begin{align*}
\lim\limits_{k \to \infty} \Big(  \mdi_{\tree^{\star}_{k,1}}(X^{(1)}) + \mdi_{\tree^{\star}_{k,1}}(X^{(2)}) \Big) 
& = \lim\limits_{k \to \infty} \Big(   \mdi_{\tree^{\star}_{k,2}}(X^{(1)}) + \mdi_{\tree^{\star}_{k,2}}(X^{(2)}) \Big)\\
& = \mathds{V}[X^{(1)}+ X^{(2)}].
	\end{align*}
	Finally, we have 
	\begin{align*}
\lim\limits_{k \to \infty} \Big(  \mdi_{\tree^{\star}_{k,1}}(X^{(1)}) - \mdi_{\tree^{\star}_{k,2}}(X^{(1)}) \Big) & =  \lim\limits_{k \to \infty} \Big(  \mdi_{\tree^{\star}_{k,2}}(X^{(2)}) - \mdi_{\tree^{\star}_{k,1}}(X^{(2)}) \Big) \\
& = \frac{1}{3} - \frac{1}{3} \Big( \frac{1}{4}\Big)^{\beta}.
	\end{align*}	
\end{proof}

\begin{proof}[Proof of Technical Lemma~\ref{thm:mdi_splimpification_critere}]
	
	With the notations of Technical Lemma~\ref{thm:mdi_splimpification_critere}, we have
	\begin{align*}
	L^{\star}_A(\ell,s) & = \V[ Y | \bX \in A ] - \P[ \bX \in A_L \,|\, \bX \in A ]  ~\V [Y | \bX \in A_L ] \nonumber \\
	& \qquad - \P[ \bX \in A_R \,|\, \bX \in A ] ~\V [Y | \bX \in A_R ]\\
	& =  \V[ m(\bX) | \bX \in A ] - \P[ \bX \in A_L \,|\, \bX \in A ]  ~\V [m(\bX) | \bX \in A_L ] \nonumber \\
	& \qquad - \P[ \bX \in A_R \,|\, \bX \in A ] ~\V [m(\bX) | \bX \in A_R ],
	\end{align*}
	since the noise $\varepsilon$ is independent of $\bX$ and the probabilities $\P[ \bX \in A_L \,|\, \bX \in A ]$ and $\P[ \bX \in A_R \,|\, \bX \in A ]$ sum to one. Therefore, 
	\begin{align}
	L^{\star}_A(\ell,s) & = \E[ m^2(\bX) | \bX \in A ] - \P[ \bX \in A_L \,|\, \bX \in A ]  ~\E [m^2(\bX) | \bX \in A_L ] \nonumber \\
	& \qquad - \P[ \bX \in A_R \,|\, \bX \in A ] ~\E [m^2(\bX) | \bX \in A_R ] \nonumber \\
	& \qquad - \mathds{E}^2[m(\bX)|\bX\in A] + \mathds{P}[\bX\in A_L | \bX \in A] \mathds{E}^2[m(\bX) |\bX\in A_L] \nonumber \\
	& \qquad + \mathds{P}[\bX\in A_R | \bX \in A] \mathds{E}^2[m(\bX) |\bX\in A_R], \label{eq:finish1}
	\end{align}
	where
	\begin{align}
	& \P[ \bX \in A_L \,|\, \bX \in A ]  ~\E [m^2(\bX) | \bX \in A_L ] \nonumber  + \P[ \bX \in A_R \,|\, \bX \in A ] ~\E [m^2(\bX) | \bX \in A_R ] \nonumber \\
	& = \E[ m^2(\bX) \mathds{1}_{\bX \in A_L} + m^2(\bX) \mathds{1}_{\bX \in A_R} \,|\, \bX \in A] \nonumber \\
	& = \E[ m^2(\bX) \,|\, \bX \in A], \label{eq:finish2}
	\end{align}
	which gives the conclusion by injecting (\ref{eq:finish2}) into (\ref{eq:finish1}). 
\end{proof}

%\bibliographystyle{plainnat} % style alphabétique en français
%% En anglais
\bibliographystyle{plainnat} % style alphabétique en anglais
%\bibliographystyle{plain} % style numéroté en anglais
% Il y a plein d'autres possibilités
% Fabrication de la biblio

\bibliography{biblio-these}

\end{document}